\documentclass[12pt,reqno,twoside]{amsart}
\usepackage{amsmath, amssymb,latexsym}
\usepackage[all,cmtip]{xy}

\newtheorem{thm}{Theorem}[section]
\newtheorem{lem}[thm]{Lemma}
\newtheorem{prop}[thm]{Proposition}
\newtheorem{cor}[thm]{Corollary}
\newtheorem{defn}[thm]{Definition}

\newtheorem{question}{Question}
\newtheorem{obs}[thm]{Observation}

\theoremstyle{remark}

\numberwithin{equation}{section}

\newcommand{\dom}{{\text{dom}}}

\renewcommand{\cor}{{\mathrm{c}}}

\newcommand{\mf}{{\mathrm{mf}}}

\newcommand{\cf}{{\mathrm{cf}}}

\newcommand{\gen}{{\mathrm{g}}}
\newcommand{\ngen}{{\mathrm{ng}}}

\begin{document}
\title[Density-1-bounding and Quasiminimality]{Density-1-bounding and
  quasiminimality in the generic degrees} \author{Peter Cholak,
  Gregory Igusa}

\thanks{This
  work was partially supported by a grant from the Simons Foundation
  (\#315283 to Peter Cholak).  Igusa was partially supported by
  EMSW21-RTG-0838506.}

\maketitle

\begin{abstract}

  We consider the question ``Is every nonzero generic degree a
  density-1-bounding generic degree?'' By previous results \cite{I2}
  either resolution of this question would answer an open question
  concerning the structure of the generic degrees: A positive result
  would prove that there are no minimal generic degrees, and a
  negative result would prove that there exist minimal pairs in the
  generic degrees.

  We consider several techniques for showing that the answer might be
  positive, and use those techniques to prove that a wide class of
  assumptions is sufficient to prove density-1-bounding.

  We also consider a historic difficulty in constructing a potential
  counterexample: By previous results \cite{I1} any generic degree
  that is not density-1-bounding must be quasiminimal, so in
  particular, any construction of a non-density-1-bounding generic
  degree must use a method that is able to construct a quasiminimal
  generic degree. However, all previously known examples of
  quasiminimal sets are also density-1, and so trivially
  density-1-bounding. We provide several examples of non-density-1
  sets that are quasiminimal.

  Using cofinite and mod-finite reducibility, we extend our results to
  the uniform coarse degrees, and to the nonuniform generic degrees.
  We define all of the above terms, and we provide independent
  motivation for the study of each of them.

  Combined with a concurrently written paper of Hirschfeldt, Jockusch,
  Kuyper, and Schupp \cite{HJKS}, this paper provides a
  characterization of the level of randomness required to ensure
  quasiminimality in the uniform and nonuniform coarse and generic
  degrees.

\end{abstract}

\section{Introduction}

Generic computability was introduced by Kapovich, Miasnikov, Schupp
and Shpilrain \cite{KMSS} as a computability-theoretic analogue of the
real-world phenomenon in which a problem is apparently much easier to
solve than would be suggested by complexity theory. The idea of
generic-case complexity is to measure the complexity of the majority
of instances of a problem, while disregarding ``difficult'' instances
if they are sufficiently uncommon.  Generic computability as well as coarse
computability, a similarly defined notion, were later studied by Jockusch and Schupp \cite{JS} in
the framework of the theory of computability theory.
% Jockusch and Schupp showed that the ``dual'' notion of coarse
% computability was
% equally natural.

% Generic computability and coarse computability were introduced by
% Jockusch and Schupp as computability-theoretic analogues of the
% real-world phenomenon in which a problem is apparently much easier to
% solve than would be suggested by complexity theory. The idea of
% generic-case complexity is to measure the complexity of the majority
% of instances of a problem, while disregarding ``difficult'' instances
% if they are sufficiently uncommon.

In generic and coarse computability, we think of a real $A$ that we
are trying to compute as the problem, and the bits of $A$ as the
instances of the problem. The goal, then, is to compute the majority
of the bits of $A$. In generic computability, we are not allowed to
make any mistakes, but we are allowed to not always give answers. In
coarse computability, we must give answers everywhere, but we are
allowed to make some mistakes. Coarse computability can be thought of
as the analogue of algorithms that take shortcuts, sacrificing
accuracy for speed, while generic computability is the analogue of
completely accurate algorithms that run very quickly in most cases,
but more slowly or perhaps not at all in others. Following the
notation of Jockusch and Schupp \cite{JS}, we make this precise as
follows.

\begin{defn}\label{real}
  \rm Let $A\subseteq\omega$. Then $A$ is \emph{density-1} if the
  limit of the densities of its initial segments is 1, or in other
  words, if $\lim_{n\rightarrow\infty}\frac{|{A\cap n}|}{n}=1$.
\end{defn}

In this paper a real $A$ is thought of either as a subset of $\omega$,
or as a function $A:\omega\rightarrow \{0,1\}$. In situations where it
will not cause confusion, the two notations will be used
interchangably, so $n\in A$ means the same thing as $A(n)=1$, and
$n\notin A$ means the same thing as $A(n)=0$.

\begin{defn}
  \normalfont A \emph{partial computation of a real $A$} is a partial
  computation $\phi$ (potentially with an oracle) such that for any
  $n$, if $\phi(n)\downarrow$, then $\phi(n)=A(n)$.
\end{defn}

\begin{defn}
  \normalfont A real $A$ is \emph{generically computable} if there
  exists a partial computable function $\phi$ such that $\dom(\phi)$ is
  density-1, and $\phi$ is a partial computation of $A$.

  % A real $A$ is \emph{generically computable} if there exists a
  % partial computable function $\phi$ with the following properties:

  % \begin{itemize}
  % \item $\dom(\phi)$ is density-1,
  % \item $\ran(\phi)\subseteq\lbrace 0,1\rbrace$,
  % \item $\phi(n)=A(n)$, for all $n\in\dom(\phi)$.

  % \end{itemize}
\end{defn}

\begin{defn}
  \rm A real $A$ is \emph{coarsely computable} if there exists a total
  computable function $\phi$, whose range is contained in
  $\lbrace 0,1\rbrace$ such that $\lbrace n\, |\, \phi(n)=A(n)\rbrace$
  is density-1.

\end{defn}

In order to obtain degree structures for these, we need to make sure
that our notion of relative computability is transitive, or in other
words, that if $X\leq Y\leq Z$, then $X\leq Z$. The outputs of our
``computations'' are generic and coarse descriptions of our reals, and
so the inputs of our computations should also be generic and coarse
descriptions. This is fairly straightforward to define for coarse
reducibility.

\begin{defn}
  \rm Let $A$ and $B$ be reals. Then \emph{$B$ is (uniformly) coarsely
    reducible to $A$} if there exists a Turing functional $\phi$ such
  that for any $C$ for which $\{n:A(n)=C(n)\}$ is density-1, $\phi^C$
  is a coarse computation of $B$. In this case, we write
  $B\leq_{\cor}A$.
\end{defn}

In nonuniform coarse reducibility, the functional $\phi$ is allowed to
depend on $C$. There is no major reason to prefer the uniform or
nonuniform reducibilities, although it appears that for coarse
reducibility, the nonuniform version is slightly easier to work with,
while for generic reducibility, the uniform version is slightly easier
to work with. This paper is primarily focused on generic reducibility,
but will occasionally derive conclusions about (uniform) coarse
reducibility from arguments concerning (uniform) generic reducibility,
and so in this paper, unless specifically specified otherwise, all
reducibilities are assumed to be uniform.

Generic reducibility is somewhat more difficult to define than coarse
reducibility, because to define generic reducibility, we are forced to
discuss what it means to use a partial computation as an oracle in a
computation. Our generic computations are not even required to tell us
whether or not they will give an output on a given value, and so our
generic computations must be able to work with oracles that also do
not tell them which outputs they will give. For this reason, we first
define partial oracles, and discuss what it means for a Turing
functional to work with a partial oracle.

\section{Partial Oracles}

We wish to define computations with partial oracles so that the
following happens: We must be able to ask questions of our oracles and
make decisions based on the outputs. Second, we must be able to avoid
being paralyzed when an oracle does not give outputs, and we must also
not be able to know whether or not the oracle will give an output in
the future, if it has not yet given an output. We may formalize this
either with time delays built in to our oracles, or with enumeration
operators, which are designed to be blind to exactly the sorts of
information that we do not wish to be able to use.

For uniform reducibilities, these are equivalent in terms of what is
reducible to what, although not necessarily in terms of the actual
procedures \cite{I1}. For nonuniform reducibilities, it is not known
whether or not the two ways to approach partial oracles are
equivalent. In this paper, we present the reducibility in terms of
time-dependent oracles.

\begin{defn}

  \normalfont

  Let $A$ be a real. Then a (time-dependent) \emph{partial oracle,}
  $(A)$, for $A$ is a set of ordered triples $\langle n,x,l\rangle$
  such that:

  $n\in\omega, \ \ x\in 2, \ \ l\in\omega$,

  $\exists l\big(\langle n,0,l\rangle\in (A)\big)\Longrightarrow
  n\notin A$,

  $\exists l\big(\langle n,1,l\rangle\in (A)\big)\Longrightarrow n\in
  A$,

  For every $n,x$, there exists at most one $l$ such that
  $\langle n,x,l\rangle\in (A)$.

\end{defn}

When using such an oracle, querying whether or not $n\in A$ consists
of initiating a search for some value of $l$ such that either
$\langle n,0,l\rangle\in (A)$, or $\langle n,1,l\rangle\in (A)$. If
there exists such an $l$, we say that $(A)(n)$ converges, or that
$(A)(n)\downarrow$. The domain of $(A)$, written $\dom((A))$, is the
set of $n$ for which there exists such an $l$. If
$\langle n,x,l\rangle\in (A)$, then we write that $(A)(n)=x$. Other
computations, processes, and queries may be carried out while
searching for such an $l$. When working with reducibilities that use
partial oracles, we will frequently abuse notation and refer to $A$ as
the partial oracle for $A$ that converges immediately on all inputs.

In previous work \cite{DI,I1,I2}, the second author omitted the
uniqueness requirement on $l$. Including this convention does not
change any reducibilities that we will define, and it will be
convenient to us in Section 6, as it will ensure that for any $n$,
$(A)\upharpoonright n$ must always be computable.

\begin{defn}
  \normalfont

  Let $A$ be a real. Then a \emph{generic oracle,} $(A)$, for $A$ is a
  partial oracle for $A$ such that $\dom((A))$ is density-1.
\end{defn}

Note that a generic computation of $A$ is the same thing as a
computation of a generic oracle for $A$.

\begin{defn}
  \normalfont Let $A$ and $B$ be reals. Then \emph{$B$ is (uniformly)
    generically reducible to $A$} if there exists a Turing functional
  $\phi$ such that for every generic oracle, $(A)$, for $A$,
  $\phi^{(A)}$ is a generic computation of $B$. In this case, we write
  $B\leq_{\gen}A$.
\end{defn}

When working with partial oracles in a context where (1) mistakes are
not allowed to be made, and (2) one must act uniformly, it can be
shown that no advantage can be gained by actively using the time
dependence of the partial oracles (see Observation
\ref{canusetimeindependent}). It is frequently much more convenient to
work only with Turing functionals which ignore the time dependence of
their partial oracles, and in this paper, we will later be assuming
that all Turing functionals are of that form.

Those familiar with enumeration reduction will see that what we define
below is an enumeration operator on the graph of the partial function
given by a partial oracle.

\begin{defn}\label{D:timeindependent}
  \normalfont Let $\phi$ be a Turing functional. Then the
  \emph{time-independent version of $\phi$} is the (potentially
  multi-valued) functional $\psi$ such that if $X$ is any partial
  oracle, then $\psi^{X}(n)$ is defined by considering all partial
  oracles $Y$ whose domains are finite subsets of $\dom(X)$, and which
  agree with $X$ on their domains, and giving every output that
  $\phi^{Y}(n)$ would give on any of those partial oracles $Y$.

  We refer to a $\psi$ defined in this way as a \emph{time-independent
    Turing functional.}
\end{defn}

Note that this process can be carried out effectively. There are
countably many finite partial oracles agreeing with $X$, and they can
be enumerated effectively in $X$. The computations of $\phi^Y(n)$ can
be carried out in parallel. Thus if $\psi$ is a time-independent
Turing functional, then for a partial oracle $X$, the outputs of
$\psi^X$ are $\Sigma^{0,X}_1$ in much the same way that the outputs of
$\phi^X$ are $\Sigma^{0,X}_1$ for an ordinary Turing functional $\phi$
and oracle and $X$.

In the remark after Theorem \ref{embeddings} we will see why
multi-valued functionals are necessary for our purposes.

We now prove some basic facts concerning partial oracles and
time-independent functionals. Our first observation is a justification
of why these are referred to as time-independent.

\begin{obs}
  Assume $\phi$ is a time-independent functional.  Let $X$ and $Y$ be
  partial oracles that have the same domains and agree on their
  domains. (So they agree as partial functions, although perhaps with
  different $l$ values at the locations where they converge.) Then
  $\phi^X=\phi^Y$ as a (potentially partial, potentially mutli-valued)
  function.
\end{obs}

Note that this justifies the abuse of notation in which $A$ is the
partial oracle that halts everywhere. Every oracle for $A$ that halts
everywhere gives the same outputs when given to a time-independent
functional, and so it does not matter which specific one we use.

\begin{proof}
  The definition of a time-independent functional specifically ensures
  that $\phi^X$ depends only on $X$ as a partial function.
\end{proof}

% \begin{proof}
%   Assume that $\phi$ is the time-independent version of
%   $\tilde\phi$. Assume that $\phi^X(n)=x$. Let $Z$ be the partial
%   oracle such that $\dom(Z)\subseteq\dom(X)$, and such that
%   $\tilde\phi^Z(n)=x$. At some finite stage, $\tilde\phi^Z(n)$ halts
%   giving output $x$,

%   Consider the finite portion $\sigma$ of the oracle $X$ that was
%   used in this computation, and the stage, $s$ at which $\phi^X(n)$
%   halts giving value $x$. At some finite stage, $Y$ shows itself to
%   halt and have the same values as $X$ at every location where $X$
%   had given an output by stage $s$. Once that happens, $\phi^Y$
%   begins considering $\sigma$ as a partial oracle that has domain a
%   subset of $\dom(Y)$ and that agrees with $Y$ on its domain, and
%   when it sees that $\phi^\sigma(n)=x$, then $\phi^Y$

%   Consider the finite portion of the oracle $X$ that was used in
%   this computation. At some finite stage, $Y$ shows itself to halt
%   and have the same values as $X$ at every location where $X$ had
%   given an output

% \end{proof}

This next observation is the primary reason that we will use
time-independent functionals. One important use of the following
observation is that if we are checking whether $A\leq_{\gen}B$ via
$\phi$, if we check that $\phi^B$ never makes any mistakes about $A$,
then it also ensures that for every partial oracle $(B)$ for $B$,
$\phi^{(B)}$ also never makes any mistakes about $A$.

\begin{obs}\label{moreismore}
  Assume $\phi$ is a time-independent functional.  Let $X$ and $Y$ be
  partial oracles such that $\dom(Y)\subseteq\dom(X)$, and such that
  $Y$ agrees with $X$ on its domain. Then if $\phi^Y(n)=x$, then
  $\phi^X(n)=x$.
\end{obs}

Note that time-independent functionals are sometimes multivalued, and
so it is possible that $\phi^X(n)$ has more than one value and that
$\phi^Y$ does not. This observation simply notes that a ``larger''
oracle must cause the functional to give more outputs.

\begin{proof}
  Assume that $\phi$ is the time-independent version of
  $\theta$. Assume that $\phi^Y(n)=x$. This is because, at some finite
  stage, we see a finite partial oracle $Z$ agreeing with $Y$ on its
  domain, such that $\theta^Z(n)=x$. We have that
  $\dom(Y)\subseteq\dom(X)$, and that $Y$ agrees with $X$ on its
  domain. Therefore, at some stage, we will see that $Z$ is a partial
  oracle that agrees with $X$ on its domain, and at that stage, we
  will begin computing $\theta^Z(n)$.
\end{proof}

The following two observations will justify us in only using
time-independent functionals, so that the aforementioned conveniences
will be available to us.

\begin{obs}\label{notwrong}
  Let $A$ and $B$ be reals. Assume that $B\leq_{\gen}A$ via
  $\phi$. Then for any $n,x$, and for any partial oracle $(A)$ for
  $A$, if $\phi^{(A)}(n)=x$, then $B(n)=x$.

  In particular, given any $X$ and $Y$, if $\phi^Y$ is multi-valued,
  then it cannot be the case that $X\leq_{\gen}Y$ via $\phi$.
  % In particular, given any $X$ and $Y$, and any partial oracle $(Y)$
  % for $Y$, if $\phi^Y$ is multi-valued, then it cannot be the case
  % that $X\leq_{\gen}Y$ via $\phi$.
\end{obs}

\noindent{\bf Remark.} The second part of the observation is
particularly useful in that it allows us to make conclusions about
whether or not $X\leq_{\gen}Y$ via $\phi$ without knowing $X$. This is
similar to the fact that in ordinary Turing computation, if $\phi^Y$
is not total, then (without knowing $X$, one knows that) it is not the
case that $X\leq_TY$ via $\phi$. It is also similar to the fact that
if one sees that for every $Z$, there is some $n$ such that
$\phi^Z(n)\neq X(n)$, then (without knowing $Y$, one knows that) it is
not the case that $X\leq_TY$ via $\phi$. These sorts of observations
are helpful in, for example, the classical construction of a real of
minimal Turing degree.

\begin{proof}
  By definition of generic reduction, we have that, for any generic
  oracle $(A)$ for $A$, if $\phi^{(A)}(n)=x$, then
  $B(n)=x$. Furthermore, given any partial oracle $(A)$ for $A$, any
  initial segment of $(A)$ can be extended to a generic oracle for $A$
  by adding more outputs at values not yet queried (or with $l$ values
  larger than have yet been checked). That generic oracle would not be
  able to have given any incorrect outputs about $B$, and so there can
  be no finite stage at which $\phi^{(A)}$ gives any incorrect outputs
  about $B$.

  It cannot be the case that $X(n)$ has more than one value, and so if
  $\phi^{(Y)}(n)$ has more than one value, then the conclusion of the
  first part of the observation does not hold.
\end{proof}

\begin{obs}\label{canusetimeindependent}
  Let $A$ and $B$ be reals. Assume that $B\leq_{\gen}A$ via
  $\phi$. Then $B\leq_{\gen}A$ via the time-independent version of
  $\phi$.
\end{obs}

\begin{proof}
  Assume that $B\leq_{\gen}A$ via $\phi$. Let $\psi$ be the
  time-independent version of $\phi$. Let $(A)$ be any generic oracle
  for $A$. If $\phi^{(A)}(n)\downarrow$, then it halts having queried
  only finitely much of $(A)$. Let $Y$ be the finite partial oracle
  that agrees with that portion of $(A)$, and that does not halt
  anywhere else. Then $\phi^Y(n)\downarrow$, and so as a result, we
  have that $\psi^{(A)}(n)\downarrow$. Thus
  $\dom(\phi^{(A)})\subseteq\dom(\psi^{(A)})$ and because
  $\dom(\phi^{(A)})$ is density-1, we also have that
  $\dom(\psi^{(A)})$ is density-1.

  % It is clear from that definition of $\psi$ that
  % $\dom(\phi^{(B)}\subseteq\dom(\psi^{(B)})$. Therefore, because
  % $\dom(\phi^{(B)}$ is density-1, we also have that
  % $\dom(\psi^{(B)})$ is density-1.

  Furthermore, every partial oracle that agrees with $(A)$ is a
  partial oracle for $A$, so by Observation \ref{notwrong}, we have
  that if $\psi^{(A)}(n)=x$, then $B(n)=x$. Therefore $\psi^{(A)}$
  never gives any incorrect outputs about $B$, and so $\psi^{(A)}$ is
  a generic computation of $B$.
\end{proof}

\section{Relationships Between Degree Structures}

The Turing degrees embed both into the coarse and into the generic
degrees \cite{DI,HJKS,JS}. These embeddings factor through the
mod-finite and cofinite degrees, respectively \cite{DI}. These
additional degree structures will be useful in terms of analyzing the
embedded Turing degrees, as there are a number of lemmas making them
convenient and relevant.

\begin{defn}
  \normalfont Let $A$ and $B$ be reals. Then \emph{$B$ is mod-finitely
    reducible to $A$} if there exists a Turing functional $\phi$ such
  that for any $C\equiv A$ (mod finite), $\phi^C$ is total, and
  computes a set that is $\equiv B$ (mod finite). In this case, we
  write $B\leq_{\mf}A$.
  % Let $A$ and $B$ be reals. Then \emph{$A$ is mod-finitely reducible
  % to $B$} if there exists a Turing functional $\phi$ such that for
  % any $C$ for which $\{n:B(n)=C(n)\}$ is cofinite,
  % $\{n:A(n)=\phi^C(n)\}$ is cofinite. In this case, we write
  % $A\leq_{\mf}B$.
\end{defn}

\begin{defn}
  \normalfont

  Let $A$ be a real. Then a \emph{cofinite oracle,} $(A)$, for $A$ is
  a partial oracle for $A$ such that $\dom((A))$ is cofinite.
\end{defn}

\begin{defn}
  \normalfont Let $A$ and $B$ be reals. Then \emph{$B$ is cofinitely
    reducible to $A$} if there exists a Turing functional $\phi$ such
  that for every cofinite oracle, $(A)$, for $A$, $\phi^{(A)}$ is a
  partial computation of $B$ with cofinite domain. In this case, we
  write $B\leq_{\cf}A$.
\end{defn}

Note that all the results from the previous section concerning
time-independent functionals in generic reduction apply equally well
when working with cofinite reduction.

Note also that cofinitely or mod-finitely computing a real is
traditionally equivalent to computing the real, since the finite error
can be directly coded into the machine. The difference here comes from
the demand that the reduction is uniform, a single reduction that
works over all potential oracles. The implications between these two
reducibilities and Turing reducibility are as follows.

\begin{thm}[Dzhafarov, Igusa]\label{DIimplication}
  $B\leq_{\mf}A\Rightarrow B\leq_{\cf}A\Rightarrow B\leq_{T}A$, and
  all of the implications are strict.

\end{thm}

The embeddings between the Turing, generic, coarse, cofinite, and
mod-finite degrees can be induced by the following maps on reals.

\begin{defn}
  If $X\subseteq\omega$, then
  $\mathcal{R}(X)=\{(2m+1)2^n:m\in\omega, n\in X\}$.

  If $X\subseteq\omega$, then
  $\widetilde{\mathcal{R}}(X)=\{m:(\exists n\in X)(2^n\leq
  m<2^{n+1})\}$.
\end{defn}

The idea behind $\mathcal{R}$ is that each bit of $X$ is coded
redundantly over infinitely many bits of $\mathcal{R}(X)$ (in fact,
positive density many bits). On the other hand, for
$\widetilde{\mathcal{R}}$, each bit of $X$ is coded into progressively
larger and larger (finite) initial segments of
$\widetilde{\mathcal{R}}(X)$. These maps induce embeddings as follows.

\begin{thm}[Dzhafarov, Igusa \cite{DI}]\label{embeddings}
  The map $X\mapsto\mathcal{R}(X)$ induces an embedding of the Turing
  degrees into either the mod-finite or cofinite degrees.

  The map $X\mapsto\widetilde{\mathcal{R}}(X)$ induces an embedding of
  the mod-finite degrees into the coarse degrees or of the cofinite
  degrees into the generic degrees.

  Symbolically, we have that for any reals $A$ and $B$:

  $(B\leq_TA)\Leftrightarrow
  (\mathcal{R}(B)\leq_\mf\mathcal{R}(A))\Leftrightarrow
  (\mathcal{R}(B)\leq_\cf\mathcal{R}(A))$.

  $(B\leq_\mf A)\Leftrightarrow
  (\widetilde{\mathcal{R}}(B)\leq_\cor\widetilde{\mathcal{R}}(A))$

  $(B\leq_\cf A)\Leftrightarrow
  (\widetilde{\mathcal{R}}(B)\leq_\gen\widetilde{\mathcal{R}}(A))$

\end{thm}

We describe the idea behind the embedding of the cofinite degrees into
the generic degrees. The embeddings of the mod-finite degrees into the
coarse degrees and of the Turing degrees into the cofinite degrees are
similar, although the embedding of the Turing degrees into the
mod-finite degrees is somewhat more subtle. See Prop 3.3 and Lemma 3.4
from \cite{DI} for a more thorough explanation.

In essence, a generic oracle for $\widetilde{\mathcal{R}}(X)$ has precisely the same information in it as a cofinite oracle for $X$. This is because for cofinitely many $n$, there must be some $m$ between $2^n$ and $2^n+1$ in the domain of a generic oracle, or else its domain cannot be density-1. Conversely, a cofinite oracle for $X$ can compute a cofinite (and hence generic) oracle for $\widetilde{\mathcal{R}}(X)$. The embedding of the mod-finite degrees into the coarse degrees is by a voting algorithm that must eventually give correct answers, and the embedding of the Turing degrees into the cofinite degrees is by an unbounded search that must eventually halt for every $n$.\\

\noindent {\bf Remark.} The algorithm described above illustrates the
reason that time-independent functionals are by their nature
potentially multivalued. If we were to use the functional described
above for computing $X$ from a generic oracle for
$\widetilde{\mathcal{R}}(X)$, but if we were to give, as input, a real
$Y$ that was not in the range of $\widetilde{\mathcal{R}}$, then there
would be intervals $[2^n,2^n+1)$ on which the oracle gave more than
one different output, and so our algorithm would also give more than
one different output on those values of $n$.

If we did not demand that our algorithms were time-independent, then we could give the \emph{first} output that we saw from our oracle, but then our output would depend on the order in which our oracle gave its outputs.\\

% The specifics of the embeddings concerning cofinite and generic
% reductions will be relevant to this paper, so we sketch the proofs.

% \begin{proof}
%   Given a generic oracle for $\widetilde{\mathcal{R}}(X)$, we can
%   cofinitely compute $X$ by, for each $n$, waiting for our oracle to
%   halt on some $m$ between $2^n$ and $2^n+1$, and then halting and
%   giving the same output. If the domain of the oracle is density-1,
%   then this computation converges for cofinitely many $n$, always
%   correctly because generic oracles never make mistakes. Likewise,
%   given a cofinite oracle for $X$ we can cofinitely and therefore
%   generically compute $\widetilde{\mathcal{R}}(X)$ in the obvious
%   fashion.

%   Given a cofinite oracle for ${\mathcal{R}}(X)$, we can compute $X$
%   by, for each $n$, searching for an $m$ such that our oracle halts
%   on $(2m+1)2^n$, and then halting and giving the same output. If
%   the domain of the oracle is cofinite, then this computation is
%   total, and again, always correct.
% \end{proof}

Using Theorem \ref{embeddings}, we may embed the Turing degrees into
any of the other four degree structures discussed in this section. In
any of these degree structures, we define a real to be
``quasiminimal'' if it is nonzero, but is not an upper bound for any
embedded Turing degrees.

\begin{defn}
  \normalfont Let $A$ be a real, and let $\leq$ be any of:
  $\leq_\cor, \leq_\gen, \leq_\mf, \leq_\cf$.

  Then $A$ is \emph{quasiminimal} in the $\leq$ degrees if $A\nleq 0$,
  and if, for every $B$, if $B\nleq_T0$, then,
  \begin{itemize}
  \item (if $\leq$ is either $\leq_\cor$ or $\leq_\gen$)
    $\widetilde{\mathcal{R}}({\mathcal{R}}(B))\nleq A$.
    % if $\leq$ is either $\leq_\cor$ or $\leq_\gen$, then
    % $\mathcal{R}(\widetilde{\mathcal{R}}(B))\nleq A$.
  \item (if $\leq$ is either $\leq_\mf$ or $\leq_\cf$)
    $\mathcal{R}(B)\nleq A$.
    % if $\leq$ is either $\leq_\mf$ or $\leq_\cf$, then
    % $\mathcal{R}(B)\nleq A$.
  \end{itemize}
  In these degree structures, a degree is quasiminimal if any
  (equivalently all) of its elements are quasiminimal.

\end{defn}

In the next section, we will see additional motivation as to what
makes quasiminimality interesting, but the basic idea is that a
quasiminimal degree is a degree that, on one hand is not computable,
but on the other hand, does not contain any actual information, in
that there are no noncomputable reals that it can compute.

The implication in Theorem \ref{DIimplication} allows quasiminimality
to propagate in a surprisingly robust manner.

\begin{prop}\label{qmpropagation}
  Assume $A$ is quasiminimal in the cofinite degrees. Then $A$ is
  quasiminimal in the mod-finite, generic, and coarse degrees.

\end{prop}

\begin{proof}
  Let $A$ be quasiminimal in the cofinite degrees, and let $B$ be
  noncomputable in the Turing degrees.

  By definition of quasiminimality, $\mathcal{R}(B)\nleq_\cf A$. By
  Theorem \ref{DIimplication}, we therefore have that
  $\mathcal{R}(B)\nleq_\mf A$.  Thus for any noncomputable $B$, we
  have that $\mathcal{R}(B)\nleq_\mf A$, and so $A$ is quasiminimal in
  the mod-finite degrees.

  Using Theorem \ref{embeddings} and the above two statements, we have
  that
  $\widetilde{\mathcal{R}}(\mathcal{R}(B))\nleq
  \widetilde{\mathcal{R}}(A)$
  in either the generic or coarse degrees, so it remains to show that
  $A\leq \widetilde{\mathcal{R}}(A)$ in both the coarse and generic
  degrees, because then the fact that
  $\widetilde{\mathcal{R}}(\mathcal{R}(B))\nleq
  \widetilde{\mathcal{R}}(A)$
  will also show that
  $\widetilde{\mathcal{R}}(\mathcal{R}(B))\nleq A$.

  This last fact follows from the idea behind the proof of Theorem
  \ref{embeddings}, mentioned in this paper under the statement of the
  theorem. A generic oracle for $\widetilde{\mathcal{R}}(A)$ contains
  enough information to cofinitely (and hence generically) compute $A$
  in a uniform manner. Likewise a coarse oracle for
  $\widetilde{\mathcal{R}}(A)$ contains enough information to
  mod-finitely (and hence coarsely) compute $A$, again uniformly.
\end{proof}

\section{Quasiminimality and Density-1 Bounding}

The results in this paper are motivated in large part by the following
question, posed by the second author as Question 3 in \cite{I2}.

\begin{question}\label{q1}

  Is it true that for every nonzero generic degree {\textbf a} there
  exists a nonzero generic degree {\textbf b} such that
  ${\textbf b}\leq_g{\textbf a}$ and such that ${\textbf b}$ is the
  generic degree of a density-1 real?

\end{question}

The resolution of this question would provide insight into the
structure of the generic degrees: If the answer is ``yes,'' then there
are no minimal generic degrees, and if the answer is ``no,'' then
there are minimal pairs in the generic degrees \cite{I2}. Also, the
question can be rephrased as a question about the relationship between
generic computability and coarse computability: We see below that a
generic degree is the degree of a density-1 real if and only if it is
the generic degree of a coarsely computable real, so the question is
about how ubiquitous the coarsely computable reals are at the bottom
of the generic degrees.

\begin{obs}\label{coarseisd1}
  Let $B$ be a real, then $B$ is coarsely computable if and only if
  the generic degree of $B$ has a density-1 set.
\end{obs}

\begin{proof}

  If $B$ is density-1, then $B$ is coarsely computable because it
  agrees with $\mathbb{N}$ on a set of density 1.

  Conversely, if $B$ is coarsely computable, then fix a computable $X$
  such that $B$ agrees with $X$ on a set of density 1. Let
  $Y=\{n:X(n)=B(n)\}$. Then $Y\leq_\gen B$ because $X$ is computable,
  and any generic oracle for $B$ can enumerate a density-1 set of
  locations where it agrees with $X$. Likewise, $B\leq_\gen Y$ because
  again, $X$ is computable, and any generic oracle for $Y$ can
  enumerate a density-1 set of locations where $X$ is correct about
  $B$, and then output the values of $X$ on those locations.
\end{proof}

% -%-%

To help us study Question \ref{q1}, we introduce the following
terminology.

\begin{defn}
  \normalfont

  A generic degree {\textbf a} is \emph{density-1-bounding} if there
  is a nonzero generic degree {\textbf b} such that
  ${\textbf b}\leq_g{\textbf a}$ and such that ${\textbf b}$ is the
  generic degree of a density-1 real.

  A real $A$ is density-1-bounding if it is of density-1-bounding
  generic degree.

\end{defn}

In prior work, the second author showed that every noncomputable real
can generically compute a density-1 real that is not generically
computable \cite{I1}. In our context, this can be rephrased as saying
that every nonzero embedded Turing degree is density-1-bounding in the
generic degrees. In particular, this implies the following.

\begin{prop}\label{old}\cite{I1}
  If {\textbf b} is not quasiminimal, then {\textbf b} is
  density-1-bounding.
\end{prop}

We will reprove this proposition in this paper, as the proof can be
modified to prove slightly more. One of the important consequences of
this proposition, however, is that any attempt to construct a {\textbf
  b} that is \emph{not} density-1-bounding must be a construction that
is capable of producing a quasiminimal {\textbf a}. Unfortunately,
currently, every example of a construction of a quasiminimal real $A$
constructs $A$ to be both quasiminimal, and also density-1, and
therefore trivially density-1-bounding.

To help us understand quasiminimality, and therefore what sorts of
constructions might potentially be able to produce sets that are not
density-1 bounding, we provide several examples of quasiminimal sets
that are not density-1. We also prove that all of the sets we
construct are density-1-bounding, and prove that a few additional
sorts of sets are density-1-bounding. This can be taken as evidence
toward the answer to Question \ref{q1} being ``yes.''

In the remainder of this section, we prove Proposition \ref{old}, and
we modify the proof to show that if {\textbf a} is an embedded
cofinite degree, then it is density-1-bounding, and also to show a
rather curious result linking non-density-1-bounding with all
$\Pi^0_1$-basis theorems. In the next section, we show that 1-generics
and 1-randoms are both quasiminimal in the generic degrees, and also
that they are density-1-bounding.

For the constructions in the rest of this section, we will use the
following notation.

\begin{defn}\label{pi}
  \normalfont Let $P_i=[2^i,2^{i+1})$. Let $X\subseteq\omega$.

  % If $e\leq i$, say that \emph{$X$ has a gap of size $2^{-e}$ at
  % $P_i$} if none of the last $2^{i-e}$-many elements of $P_i$ are
  % elements of $X$.
  If $e\leq i$, say that \emph{$X$ has a gap of size $2^{-e}$ at
    $P_i$} if $|X\cap P_i|\leq 2^i-2^{i-e}$.

  % If $e\leq i$, say that \emph{$X$ has a gap of size $2^{-e}$ at
  % $P_i$} if at least $2^{i-e}$-many elements of $P_i$ are not
  % elements of $X$.
\end{defn}

The following lemma illustrates the control that these $P_i$ give over
a construction that uses them.

\begin{lem}\label{pilem}
  Let $X\subseteq\omega$.

  % Assume that the only elements missing from $X$ are from the ends
  % of the $P_i$. (i.e., that for every $i$, $X\cap P_i$ is an initial
  % segment of $[2^i,2^{i+1})$.)

  Then $X$ is density-1 if and only if for every $e$, there are at
  most finitely many $i\geq e$ such that $X$ has a gap of size
  $2^{-e}$ at $P_i$.
  % for every $e$, there are at most finitely many $i\leq e$ such that
  % $X$ has a gap of size $2^{-e}$ at $P_i$, and that $X$ is missing
  % no other elements other than those that it is missing due to those
  % gaps.
\end{lem}

% This lemma was originally proved as Lemma 2.3 in \cite{I1}, but
% there was a slight error in the proof, in that $i_e$ was used
% instead of $e+i_e$ in the second half of the proof, so the bounds
% were slightly inaccurate. We correct this here.

This is a slight strengthening of Lemma 2.3 from \cite{I1}, which
required that the gaps be at the ends of the $P_i$. We will require
this lemma in the generality just stated later in this section.

\begin{proof}
  % Assume that for some $e$, there are infinitely many $i$ such that
  % $X$ has a gap of size $2^{-e}$ at $P_i$. Then, because each $P_i$
  % contains half the elements of $\omega$ up to the end of $P_i$, we
  % have that if $X$ has a gap of size $2^{-e}$ at $P_i$, then
  % $\frac {|X\upharpoonright 2^{i+1}|}{2^{i+1}}\leq1-2^{-e-1}$. If,
  % for a single fixed value of $e$, this happens infinitely often,
  % then $\liminf_{n\rightarrow\infty}\frac{|{X\cap n}|}{n}\neq1$, and
  % so $X$ is not density-1.

  If $X$ has a gap of size $2^{-e}$ at $P_i$ then, because each $P_i$
  contains half the elements of $\omega$ up to the end of $P_i$,
  $\frac {|X\upharpoonright 2^{i+1}|}{2^{i+1}}\leq1-2^{-e-1}$. If, for
  a single value of $e$, this happens infinitely often, then $X$ is
  not density-1.

  % Conversely, assume that for every $e$, there are finitely many $i$
  % such that $X$ has a gap of size $2^{-e}$ at $P_i$. Then, for any
  % $e$, let $i_e$ be minimal such that for $i>i_e$, $X$ does not have
  % a gap of size $2^{-e}$ at $P_i$. In that case, for $i>e+i_e$, we
  % have that
  % $\frac {|X\upharpoonright
  % 2^{i+1}|}{2^{i+1}}\geq(1-2^{-e-1})(1-2^{-e})$.
  % (The $1-2^{-e-1}$ is because $i>e+i_e$, so everything up to the
  % end of $P_{i_e}$ is at most $2^{-e-1}$ of the numbers up to the
  % end of $P_i$. The $1-2^{-e}$ is because after the end of
  % $P_{i_e}$, $X$ does not have a gap of size $2^{-e}$ at each $P_j$,
  % and so it is missing less than $2^{-e}$ of the elements after the
  % end of $P_{i_e}$.) As $e$ approaches infinity, this bound
  % approaches 1.

  % If $n$ is not of the form $2^i$, then
  % $\frac {|X\upharpoonright {n}|}{{n}}$ is even larger, because the
  % only elements missing from $X$ are from the ends of the $P_i$, so
  % $\frac {|X\upharpoonright {n}|}{{n}}$ attains its local minima at
  % the ends of the $P_i$. Thus,
  % $\lim_{n\rightarrow\infty}\frac{|{X\cap n}|}{n}=1$.

  Conversely, assume that for every $e$, there are finitely many $i$
  such that $X$ has a gap of size $2^{-e}$ at $P_i$. Let
  $\epsilon>0$. We must show that there is an $m$ such that
  $\forall n\geq m\left(\frac {|X\upharpoonright
      {n}|}{{n}}>1-\epsilon\right)$.
  Choose $e\in \omega$ such that $(1-3\cdot2^{-e})>1-\epsilon$. Fix
  $j$ such that for $\forall i>j$, $X$ does not have a gap of size
  $2^{-e}$ at $P_i$. Let $m=2^{j+e+1}$. Then we claim that for
  $n\geq m$, $\frac {|X\upharpoonright {n}|}{{n}}>1-\epsilon$.

  The reason for this is that for $n\geq m$, the smallest value that
  $\frac {|X\upharpoonright {n}|}{{n}}$ could possibly take is at the
  beginning of some $P_k$, after omitting the largest number of
  elements that can be omitted from $P_k$ without causing $X$ to have
  a gap of size $2^{-e}$ at $P_k$. If this happens, then the elements
  missing from $X\upharpoonright{n}$ can, at most, consist of: these
  elements from the beginning of $P_k$, all of the elements less than
  $2^{j+1}$, and the elements missing from the $P_i$ for $j<i<k$. The
  first elements are at most $n2^{-e}$ many elements because the
  number of elements of $P_k$ is at most $n$. The second elements are
  at most $n2^{-e}$ many elements because they are at most $2^{j+1}$
  many elements, and $n\geq 2^{j+e+1}$. The last elements are at most
  $n2^{-e}$ many elements because, for each $i$ between $j$ and $k$,
  $X$ does not have a gap of size $2^{-e}$ at $P_i$.

  Thus, there are at most $3n2^{-e}$ elements missing from
  $X\upharpoonright{n}$, and so
  $\frac {|X\upharpoonright {n}|}{{n}}>(1-3\cdot2^{-e})>1-\epsilon$.
\end{proof}

\begin{proof}[Proof (Proposition \ref{old})]

  We define a pair of functionals, $\phi$, $\psi$ such that for any
  real $X$, if $X$ is not left c.e.\ (not the leftmost path of any
  computable tree) then $\phi^X$ is an enumeration of a density-1 set
  with no density-1 c.e.\ subset, and if $X$ is not right c.e.\ (not the
  rightmost path of any computable tree) then $\psi^X$ is an
  enumeration of a density-1 set with no density-1 c.e.\ subset. This
  will suffice to prove Proposition \ref{old} as follows.

  Let {\textbf a} be a generic degree that is not quasiminimal, and let $A$ have generic degree {\textbf a}. By definition of quasiminimality, fix $X_0$ noncomputable such that $\widetilde{\mathcal{R}}(\mathcal{R}(X_0))\leq_\gen A$. Because $X_0$ is noncomputable, $X_0$ must be either not left c.e.\ or not right c.e.. By symmetry, assume $X_0$ is not left c.e.. Let $B$ the set enumerated by $\phi^{X_0}$, and let {\textbf b} be the generic degree of $B$. By construction of $\phi$, $B$ is density-1. Furthermore ${\textbf b}\leq_\gen{\textbf a}$ because any generic oracle for $A$ can generically compute $\widetilde{\mathcal{R}}(\mathcal{R}(X_0))$, and this generic computation can uniformly be used to compute $X_0$, and hence to enumerate $B$. This enumeration is a generic computation of $B$ because it halts on density-1, and it is correct about $B$ wherever it halts. Finally, $B$ is not generically computable because it has no density-1 c.e.\ subset, and if there were a generic computation of $B$, then the set of $n$ where that computation halted and outputted a 1 would be a density-1 c.e.\ set. (Density-1 because it is the intersection of two density-1 sets, c.e.\ because halting and outputting a 1 is a $\Sigma_1$ condition.)\\

  We construct $\phi$ as follows. $\psi$ will be constructed
  symmetrically. For each $e$, we will have an $e$th strategy, which
  will act to ensure that if the $e$th c.e.\ set $W_e$ has density 1,
  then $W_e$ is not a subset of $\phi^X$ for any $X\in 2^\omega$. In
  doing so, there will be at most one $X=X_e\in 2^\omega$ such that
  the strategy prevents $\phi^{X_e}$ from enumerating a density-1 set,
  and that $X_e$ will be left c.e.. Lemma \ref{pilem} will then be used
  to ensure that the only reals $X$ such that $\phi^{X}$ is not
  density-1 are the $X_e$ from these strategies.

  We define $\phi$ using the subintervals $P_i$ from Definition
  \ref{pi}. At stage $s$, we simultaneously define $\phi^X$ on $P_s$
  for every $X\in 2^\omega$. Note that because $\phi^X$ enumerates its
  elements in increasing order, the set enumerated by $\phi^X$ is
  actually uniformly computable from $X$, not just uniformly
  generically computable from $X$, although this is not relevant for
  our purposes.

  At stage $s$, for each $e<s$, the $e$th strategy acts as
  follows. Consider the tree $T_{e,s}$ whose paths are the reals $X$
  such that the numbers less than $2^s$ enumerated by $\phi^X$ are a
  superset of the numbers less than $2^e$ enumerated by $W_e$ by stage
  $s$. The reals $X$ not on $T_{e,s}$ are the reals that have already
  ``beaten'' $W_e$, in that $W_e$ has enumerated an element that they
  will never enumerate, and so the $e$th strategy will never need to
  work with those $X$ again. If $T_{e,s}$ has no paths, then the $e$th
  strategy has accomplished its task, and so it does nothing.

  If $T_{e,s}$ has a path, then the $e$th strategy places a ``marker''
  $p_{e,s}$ on the shortest unmarked node of the leftmost path of
  $T_{e,s}$. Here, an unmarked node is a node such that the $e$th
  strategy has not yet placed a marker on that node, regardless of
  whether or not other strategies have marked that node. When it
  places the marker $p_{e,s}$ on the node $\sigma$, it declares that
  for every $X\succ\sigma$, $\phi^X$ must have a gap of size $2^{-e}$
  at $P_s$. Specifically, the strategy requires that the last
  $2^{s-e}$ elements of $P_s$ are not in $\phi^X$. The strategy makes
  no other requirements for $\phi^X$ if  $X\nsucc\sigma$.

  To define $\phi^X\upharpoonright P_s$ for any given $X$, we simply
  enumerate every element of $P_s$ that no strategy requires us to not
  enumerate.

  \vspace{5pt}

  We now prove that this construction has the desired properties.

  First we show that for every $e$, and for every $X\in2^\omega$,
  $W_e$ is not a density-1 subset of $\phi^X$.  Let $e$ be
  given. Either there is some $s$ such that $T_{e,s}$ has no paths, or
  for every $s$, $T_{e,s}$ has at least one path.

  In the first case we have that for every $X$, $W_e$ has enumerated a
  number not enumerated by $\phi^X$, and so for every $X$, $W_e$
  cannot possibly be a density-1 subset of the set enumerated by
  $\phi^X$.

  In the second case, let $T_e=\bigcap_s T_{e,s}$, and let $X_e$ be
  the leftmost path of $T_e$. Note that for every $\sigma\prec X_e$,
  there is some $s$ such that $\sigma$ has the marker $p_{e,s}$. This
  is because every $\tau$ to the left of that $\sigma$ eventually is
  removed from some $T_{e,s}$, and once all of those $\tau$ are
  removed, after at most $|\sigma|$ many more steps, $\sigma$ must
  have a marker placed on it by the $e$th strategy.

  All of these $\sigma$ are in $T_e$, and so, in particular, there are
  infinitely many nodes in $T_e$ with markers placed by the $e$th
  strategy. Therefore there are infinitely many $s$ for which $W_e$
  has a gap of size $2^{-e}$ at $P_s$, because if $W_e$ ever
  enumerates an element from the last $2^{s-e}$ many elements of
  $P_s$, then at that stage $s'$, every $X$ extending the node marked
  with $p_{e,s}$ will be removed from $T_{e,s'}$. Thus, each marker on
  a $\sigma\in T$ corresponds to a gap of size $2^{-e}$ in
  $W_e$. Therefore $W_e$ is not an enumeration of a density-1 set, and
  in particular, not an enumeration of a density-1 subset of the set
  enumerated by $\phi^X$ for any $X$.

  Next we show that if $X$ is not left c.e., then $\phi^X$ is an
  enumeration of a density-1 set.

  Each $X_e$, if it is defined, is left c.e.\ because it is the
  leftmost path of a $\Pi^0_1$ tree, and because for every $\Pi^0_1$
  tree, there is a computable tree having the same infinite paths. It
  remains to show that if $X$ is not equal to any of the $X_e$, then
  $\phi^X$ enumerates a density-1 set. To show this, we show that for
  every $e$, if $X_e$ is not defined, or if $X$ is not equal to $X_e$,
  then only finitely many markers are placed on initial segments of
  $X$ by strategy $e$. This will show that for every $e$, $\phi^X$ has
  finitely many gaps of size $2^{-e}$, and so by Lemma \ref{pilem},
  $\phi^X$ is density-1.

  So, fix $e$, assume $X_e$ is defined, and assume $X\neq X_e$. If $X$
  is to the left of $X_e$, then at some stage $s$, $X$ was removed
  from $T_{e,s}$, and after that stage, $X$ stopped receiving new
  markers from strategy $e$. No marker ever gets placed to the right
  of $X_e$ by strategy $e$, so if $X$ is to the right of $X_e$, then
  the only markers from strategy $e$ on $X$ are those placed on
  $\sigma$ that are initial segments of both $X$ and $X_e$, of which
  there are at most finitely many. If $X_e$ is not defined, it is
  because there is some stage at which strategy $e$ stops acting, and
  so strategy $e$ uses at most finitely many markers, so in
  particular, each $X$ can receive at most finitely many markers from
  strategy $e$.

  Thus, if $X$ is not equal to any of the $X_e$, then for every $e$
  only finitely many markers are placed on initial segments of $X$ by
  strategy $e$, so $\phi^X$ is density-1.

  To construct $\psi$, we do the same construction but using rightmost
  paths instead of leftmost paths.
\end{proof}

We now modify the proof to prove a stronger result.

\begin{prop}\label{newd1b}

  Assume $A$ is noncomputable. Then $\widetilde{\mathcal{R}}(A)$ is
  density-1-bounding in the generic degrees.

\end{prop}

Proposition \ref{old} says that embedded noncomputable Turing degrees
are density-1-bounding, while Proposition \ref{newd1b} says that
embedded noncomputable cofinite degrees are density-1-bounding. In the
next section, we will demonstrate examples of quasiminimal cofinite
degrees, which, in particular, embed into the generic degrees as
quasiminimal, and hence this proposition is strictly stronger than
Proposition \ref{old}.

\begin{proof}

  Assume $A$ is noncomputable, and further assume that $A$ is not left
  c.e.. Let $\phi$ be the functional $\phi$ from the proof of
  Proposition \ref{old}. Let $\tilde{\phi}$ be the time-independent
  functional defined from $\phi$ as follows.

  Given a partial oracle $(X)$, $\tilde\phi^{(X)}(n)\downarrow=1$ if
  and only if for every $Y$ such that $(X)$ is a partial oracle for
  $Y$, $\phi^Y(n)\downarrow=1$.

  (Note that this definition is very similar to the definition of
  $\psi$ as in Definition \ref{D:timeindependent}. The primary
  difference is that in this case, $\phi$ was created as a functional
  that uses ordinary Turing oracles, whereas in Definition
  \ref{D:timeindependent}, $\phi$ is a functional that uses partial
  oracles.)

  By compactness, this is a $\Sigma_1$ operation: if every $Y$
  agreeing with $(X)$ enumerates some $n$, then there is a finite
  stage at which we see this happen. Note also that the quantifier
  over $Y$ agreeing with $(X)$ does not interfere with the fact that
  $(X)$ enumerates its domain: As $(X)$ produces more answers, the set
  of $Y$ agreeing with $(X)$ becomes smaller, and so the set of $n$
  for which $\tilde\phi^{(X)}(n)\downarrow$ becomes larger, so we
  never need to ``take back'' any $n$ for which $\tilde\phi^{(X)}(n)$
  has halted.

  We now claim that if $(A)$ is a cofinite oracle for $A$ (recall that
  any generic oracle for $\widetilde{\mathcal{R}}(A)$ can uniformly
  produce a cofinite oracle for $A$), then $\tilde\phi^{(A)}$ is a
  generic computation of the set enumerated by $\phi^A$, which we have
  proved is density-1 and not generically computable. Because $A$ is
  one of the $Y$'s agreeing with $(A)$, we have that the domain of
  $\tilde\phi^{(A)}$ is a subset of the domain of $\phi^A$, so in
  particular, $\tilde\phi^{(A)}$ never makes any mistakes about the
  set it is computing. It remains to show that
  $\dom(\tilde\phi^{(A)})$ is density-1.

  To prove this, let $S$ be the set of initial segments of $Y$'s such
  that $(A)$ is a partial oracle for $Y$. We show that strategy $e$
  from the proof of Proposition \ref{old} places at most finitely many
  markers on elements of $S$. To show this, define $T_e$ as in the
  proof of Proposition \ref{old}. If $T_e$ is finite, then strategy
  $e$ places at most finitely many markers, and so it places at most
  finitely many markers on elements of $S$. If $T_e$ is infinite, then
  let $X_e$ be the leftmost path of $T_e$, and we claim that for every
  $\sigma\prec X_e$, strategy $e$ places at most finitely many markers
  on $\tau$ that are not extensions of $\sigma$. This is because
  everything to the left of $\sigma$ eventually gets removed from
  $T_e$, and everything to the right of $\sigma$ never gets a marker,
  and there are only finitely many things below $\sigma$.

  So it remains to show that there is some $\sigma\prec X_e$ such that
  $\sigma$ has no extensions in $S$. This follows from the fact that
  $A$ is not left c.e., $(A)$ is a cofinite oracle for $A$, and every
  $Y$ agreeing with $(A)$ is mod-finitely equal to $A$, and so also
  not left c.e.. Therefore there must be some $n$ such that
  $(A)(n)\downarrow\neq X_e(n)$, because $X_e$ is left c.e., and so
  $(A)$ is not a partial oracle for $X_e$.

  We therefore conclude that for every $e$, $S$ has only finitely many
  markers on it placed by strategy $e$, and thus that the domain of
  $\tilde\phi^{(A)}$ has at most finitely many gaps of size $2^{-e}$,
  because if the $e$th strategy places a marker $p_{e,s}$ on some
  $\sigma\notin S$, then for every $Y$ agreeing with $(A)$,
  $\sigma\nprec Y$, and so $\phi^Y$ does not have a gap of size
  $2^{-e}$ at $P_s$. (Or rather, the $e$th strategy does not cause it
  to have such a gap. The gaps of size $2^{-e}$ can be created by
  strategy $e'$ for any $e'<e$, but this proof shows that each one of
  those strategies causes at most finitely many gaps to appear in the
  domain of $\tilde\phi^{(A)}$.) Thus, by Lemma \ref{pilem}, the
  domain of $\tilde\phi^{(A)}$ is density-1.
\end{proof}

As a corollary to the proof of Proposition \ref{newd1b}, we make an
observation that will have a number of consequences.

\begin{obs}\label{weird1}

  Assume there is no left c.e.\ set $X$ such that $\{n:X(n)=A(n)\}$ is
  density-1. Then $A$ is density-1-bounding.

\end{obs}

\begin{proof}

  Define $\tilde\phi$ as above. If $A$ is a real, there is no left
  c.e.\ set $X$ such that $\{n:X(n)=A(n)\}$ is density-1, and $(A)$ is
  a generic oracle for $A$, then, in particular, for every $e$, $X_e$
  is not one of the reals $Y$ that agree with $(A)$. In particular, we
  therefore again have that there is some $\sigma\prec X_e$ such that
  no $Y\succ\sigma$ agrees with $(A)$. After some finite stage, all
  markers placed by strategy $e$ are placed on extensions of $\sigma$,
  and so therefore, there are at most finitely many markers placed on
  initial segments of $Y$'s that agree with $(A)$. This again proves
  that for every $e$, there are at most finitely many gaps of size
  $2^{-e}$ in the set enumerated by $\tilde\phi^{(A)}$.
\end{proof}

This observation can be modified by a technique from a previous paper
of the second author (Lemma 2.6 from \cite{I1}) which says that ``left
c.e.'' can be replaced by any property that can by realized by a
uniformly in $0'$ effective basis theorem.

\begin{obs}\label{weird2}

  Let $\mathcal{F}$ be any function from reals to reals such that for
  any $\Pi^0_1$ tree $T$, if $T$ is infinite, then $\mathcal{F}(T)$ is
  an infinite path through $T$ that is uniformly computable in $0'$
  together with a $\Pi^0_1$ index for $T$.

  Assume there is no $X$ in the range of $\mathcal{F}$ such that
  $\{n:X(n)=A(n)\}$ is density-1. Then $A$ is density-1-bounding.

\end{obs}

In particular, using the Low Basis Theorem, this says that if $A$ is
not density-1-bounding, then $A$ agrees with a low set on a set of
density 1. Likewise, using the cone avoidance basis theorem, this says
that if $A$ is not density-1-bounding, then for any noncomputable
$\Delta^0_2$ $B$, there is a $\Delta^0_2$ set $X$ such that
$X\ngeq_TB$ and $A$ agrees with $X$ on a set of density-1, etc. In
essence, Observation \ref{weird2} is an observation schema across $0'$
computable basis theorems.

\begin{proof}[Proof (sketch).]

  In the proof of Proposition \ref{old}, wherever strategy $e$ would
  normally place a marker on the shortest unmarked node of $T_{e,s}$,
  have it instead place a marker on the shortest unmarked node of the
  stage-$s$ $\Delta^0_2$ approximation to $\mathcal{F}(T_e)$. (A
  $\Pi^0_1$ index for $T_e$ can be obtained using the recursion
  theorem.) Let $X_e=\mathcal{F}(T_e)$. Because $\Delta^0_2$
  approximations eventually converge, for every $\sigma\prec X_e$, we
  have that eventually all markers are placed on extensions of
  $\sigma$, which is sufficient for the verification for Proposition
  \ref{old} and also for Observation \ref{weird1}.
\end{proof}

\noindent{\bf Remark.}
If the proof of Proposition \ref{old} could be modified somehow to
make each $X_e$ computable, then Observation \ref{weird1} would
probably be able to be modified to say that if there is no computable
set $X$ such that $\{n:X(n)=A(n)\}$ is density-1 then $A$ is
density-1-bounding. Note that Observation \ref{coarseisd1} implies
that if $A$ is coarsely computable, then $A$ is in the same generic
degree as a density-1 set, and so in particular is density-1-bounding.
Combining these two would prove that if $A$ is not generically
computable, then $A$ is density-1-bounding, solving Question \ref{q1}.

Simplifying the analysis above leads to a question that has been open
since the writing of \cite{I1}, but that now appears to be
sufficiently motivated to be worth stating as an open problem.

\begin{question}\label{q2}
  Is there a uniform proof of Proposition \ref{old}?
\end{question}

This question asks whether there is a single $\phi$ such that for
every $A$, if $A$ is not computable, then $\phi^A$ is an enumeration
of a density-1 set with no density-1 c.e.\ subset. A positive solution
would not necessarily answer Question \ref{q1} as well, because the
solution might not admit the modification required for Observation
\ref{weird1}, but the question is elegant and simple enough that it
might merit study in its own right.

We conclude the section with a proof of a result that effectively says
that the observations stated here are not already sufficient to prove
that every nonzero degree is density-1-bounding in the generic
degrees.

\begin{prop}[Dzhafarov, Igusa]\label{wiggletree}

  There exists a real $A$ such that for every $\mathcal{F}$ as in
  Observation \ref{weird2}, there is an $X$ in the range of
  $\mathcal{F}$ such that $\{n:X(n)=A(n)\}$ is density-1, but such
  that $A$ is neither coarsely computable nor generically computable.

\end{prop}

Note that by Observation \ref{weird2} and Lemma \ref{coarseisd1}, if
there existed a $A$ that was not density-1-bounding, then such a $A$
would necessarily need to be a witness to Proposition
\ref{wiggletree}.

This construction is a slight modification of a construction by
Dzhafarov, Igusa, and Westrick [personal communication] in which they
proved Proposition \ref{wiggletree} without the additional requirement
that $A$ not be generically computable. The previous construction
could only build generically computable $A$ although it contained most
of the ideas necessary to write this proof.

\begin{proof}
  We build an infinite computable tree $T\subseteq 2^{<\omega}$ such
  that given any two paths in $[T]$, the two paths agree on density 1,
  and such that no path in $[T]$ is is either coarsely computable or
  generically computable. Any path in $[T]$, can then be used as our
  $A$. This is because, given any $\mathcal{F}$, there must be a path
  in $[T]$ that is in the range of $\mathcal{F}$, and by construction
  of $T$, that path must agree with $A$ on density 1. Furthermore, by
  construction of $T$, $A$ is neither coarsely computable nor
  generically computable.

  The construction is as follows. We have requirements:
 
  $\mathcal{C}_i$: $\phi_i$ does not coarsely compute a path through
  $T$.
 
  $\mathcal{G}_i$: $\phi_i$ does not generically compute a path
  through $T$.

  % We first remark that if each individual requirement can be met
  % uniformly, (so there is a computable function $f$ such that
  % $\phi_{f(2i)}$ computes a tree $T_{2i}$ all of whose paths agree
  % on density 1 such that $\phi_i$ does not coarsely compute a path
  % through $T$, and such that $\phi_{f(2i+1)}$ computes a tree
  % $T_{2i+1}$ all of whose paths agree on density 1 such that
  % $\phi_i$ does not generically compute a path through $T$) then we
  % can trivially combine those trees to create $T$ by the following
  % use of the recursion theorem.

  We first remark that if each individual requirement can be met
  uniformly, then we can combine those trees to produce $T$. More
  formally:

  \vspace{5pt}

\noindent 
\emph{Claim.}  Assume there is a computable function $f$ such that
$\phi_{f(2i)}$ computes an infinite tree $T_{2i}\subseteq 2^{<\omega}$
all of whose paths agree on density 1 such that $\phi_i$ does not
coarsely compute a path through $T_{2i}$, and such that
$\phi_{f(2i+1)}$ computes an infinite tree
$T_{2i+1}\subseteq 2^{<\omega}$ all of whose paths agree on density 1
such that $\phi_i$ does not generically compute a path through
$T_{2i+1}$. Then there is an infinite computable tree
$T\subseteq 2^{<\omega}$ all of whose paths agree on density 1 such
that no path is either generically computable or coarsely computable.

\vspace{5pt}

\noindent 
\emph{Proof of Claim.}  Given $X\in 2^\omega$, let
$X_k=\{n:(2n+1)2^k\in X\}$. Note that if $X$ is generically
computable, then each $X_k$ is generically computable, and if $X$ is
coarsely computable, then each $X_k$ is coarsely
computable. %Note also that given $X,Y\in2^\omega$, $X$ and $Y$ agree on a set of density 1 if and only if for each $k$, $X_k$ and $Y_k$ agree on a set of density 1.

For each $k$, let $T_k$ be as in the statement of the claim. Define
$T$ by $X\in[T]$ iff for each $k$, $X_k\in [T_k]$. (Note that because
the $T_k$ are uniformly computable, this is a $\Pi^0_1$ class, and so
in particular, there is a $T$ whose paths are the reals such that
$\forall kX_k\in T_k$.) We then claim that if $X\in [T]$, then $X$ is
neither coarsely nor generically computable.

To see this, assume otherwise. Assume $X\in[T]$ and $\psi$ is a Turing
functional that coarsely computes $X$ (the proof for generic
computation will be nearly identical). Define $g$ so that
$\phi_{g(i)}(n)=\psi((2n+1)2^{2i})$. Note then that if $\psi$ is a
coarse computation of $X$, then $\phi_{g(i)}$ is a coarse
computation of $X_{2i}$. By the recursion theorem, there exists an $i$
such that $\phi_{g(i)}=\phi_i$, providing a contradiction, because
$X_i$ was specifically constructed to not be coarsely computable via
$\phi_i$.
% Define $f$ so that $\phi_{f(i)}$ is the coarse computation of
% $X_{2i}$ given by restricting $\psi$ to
% $\{(2n+1)2^{2i}:n\in\omega\}$

It remains to show that any two paths in $[T]$ must agree on
density-1. Let $X,Y\in [T]$, let $\epsilon>0$. Let $l$ be such that
$2^{-l}<\frac\epsilon2$. For each $k<l$, let $n_k$ be such that for
$m>n_k$, $X_k\upharpoonright m$ and $Y_k\upharpoonright m$ agree on
$m(1-\frac\epsilon2)$ many bits. Let $n=\max_{k<l}(n_k2^k)$. It is
straightforward to check that for $m>n$, $X\upharpoonright m$ and
$Y\upharpoonright m$ agree on $m(1-\epsilon)$ many bits.

This concludes the proof of the claim. To complete the proof of the
theorem, it now remains to construct the $T_k$ uniformly.

\vspace{5pt}

\noindent
\emph{Meeting requirement $\mathcal{C}_i$:} We meet all of these
requirements uniformly with a single tree. Let $X_0$ be any real that
is generically computable but not coarsely computable (exists by
\cite{JS}).

Let $\phi$ be the Turing functional that generically computes
$X_0$. For every $i$, let $T_{2i}$ be the tree of all $X$ such that
$\phi$ is a generic computation of $X$. Note that all such $X$ agree
on density 1 (because they agree on $\dom(\phi)$), and are not
coarsely computable (because a coarse computation of one of them would
be a coarse computation of every one of them, and $X_0$ is not
coarsely computable).

% We construct a tree $T_{2i}\subseteq 2^{<\omega}$ all of whose paths
% agree on density 1 such that $\phi_i$ does not coarsely compute a
% path through $T_{2i}$. In fact, we will ensure that all the paths in
% $[T_{2i}]$ agree on a cofinite set.

% At stage $0$, we put the empty string and the string
% $\langle 0\rangle$ into $T_{2i}$.

% At stage $s>0$, we decide whether or not $\sigma\in T_{2i}$ for all
% $\sigma\in 2^{<\omega}$ with $2^{s-1}<|\sigma|\leq2^s$ by the
% following algorithm. If there is only one $\tau$ with
% $|\tau|=2^{s-1}$ in $T_{2i}$, then for that $\tau$, for every
% $\sigma\succ\tau$ with $2^{s-1}<|\sigma|\leq2^s$, we put $\sigma$ in
% to $T$.

% If there is more than one $\tau$ with $|\tau|=2^{s-1}$, then the
% construction will ensure that there is a unique $k<s$ such that:

% \begin{itemize}
% \item There is only one $\tau_0\in T_{2i}$ such that $|\tau_0|=2^k$
%   and $\tau_0$ extends to an element of length $2^{s-1}$ in
%   $T_{2i}$.
% \item For every $\sigma\succ\tau_0$ with $|\sigma|=2^{k+1}$,
%   $\sigma$ extends to a unique element of length $2^{s-1}$ in
%   $T_{2i}$.

% \end{itemize}

% For that value of $k$, we ask whether, at stage $s$, $\phi_i$ has
% halted on all inputs $\leq2^{k+1}$.

\vspace{5pt}

\noindent
\emph{Meeting requirement $\mathcal{G}_i$:} We construct a tree
$T_{2i+1}\subseteq 2^{<\omega}$ all of whose paths agree on density 1
such that $\phi_i$ does not generically compute a path through
$T_{2i+1}$.

We construct a computable tree $\widetilde{T}\subseteq 2^{<\omega}$
such that any two paths in $[\widetilde{T}]$ agree on density 1 but
such that, for any $n\in\omega$, there is at least one
$X\in[\widetilde{T}]$ with $X(n)=0$ and at least one
$X\in[\widetilde{T}]$ with $X(n)=1$. If we ever see $\phi_i(n)$ halt
for any value of $n$, then we let $n_i$ be the first value of $n$ for
which $\phi_i(n)\downarrow$, and $s_i> n_i$ be the number of stages
required to see that $\phi_i(n_i)\downarrow$.

We then
define:
$$T_{2i+1}=\{\sigma\in\widetilde{T}:|\sigma|<s_i\}\cup\{\sigma\in\widetilde{T}:|\sigma|\geq
s_i\wedge\sigma(n_i)\neq\phi_i(n_i)\}.$$
Note, in particular, that $T_{2i+1}$ will be defined uniformly in $i$,
and that if $\phi_i$ never halts (or if $\phi_i(n_i)\notin\{0,1\}$)
then $T_{2i+1}=\widetilde{T}$.

\vspace{5pt}

To ensure that all paths in $[\widetilde{T}]$ agree on density 1, we
ensure that every path in $[\widetilde{T}]$ is density-1 (as a subset
of $\omega$), and so any two paths agree on a set containing their
intersection, which is a density-1 set.

% We put the empty string and the strings $\langle 0\rangle$ and
% $\langle 1\rangle$ into $\widetilde{T}$.
We put the empty string into $\widetilde{T}$.

% When we define
% $\widetilde{T}\upharpoonright\{\sigma:2^{n}<|\sigma|\leq2^{n+1}\}$
% we will have that there exist exactly $2^{n+1}$-many
% $\sigma\in\widetilde{T}$ with $|\sigma|=2^n$
When we define
$\widetilde{T}\upharpoonright\{\sigma:2^{n}\leq|\sigma|<2^{n+1}\}$ we
will have that there exist exactly $2^{n}$-many
$\sigma\in\widetilde{T}$ with $|\sigma|=2^n-1$, and we will ensure
that there exist exactly $2^{n+1}$-many $\sigma\in\widetilde{T}$ with
$|\sigma|=2^{n+1}-1$.

To define
$\widetilde{T}\upharpoonright\{\sigma:2^{n}\leq|\sigma|<2^{n+1}\}$,
each $\sigma\in\widetilde{T}$ with $|\sigma|=2^n-1$ selects one
element $m_\sigma$ of $[2^n,2^{n+1})$, and extends itself so that the
extensions of $\sigma$ of length $2^{n+1}-1$ are precisely the two
$\tau$ of that length such that $\tau\succ\sigma$ and $\tau(m)=1$ if
($m\in[2^n,2^{n+1})$ and $m\neq m_\sigma$).

If each $\sigma$ selects a different $m_\sigma$, then every
$m\in[2^n,2^{n+1})$ will be selected, because
$|[2^n,2^{n+1})|=2^n$. This is easy to arrange.

At the end of the construction, for each $X\in\widetilde{T}$, for each
$n\in\omega$, there will be at most one $m\in[2^n,2^{n+1})$ such that
$m\notin X$, and so $X$ will be density-1. Also, every $\sigma$ in
$\widetilde{T}$ extends to an $X\in[\widetilde{T}]$, and so, in
particular, for every $m$, there is an $X\in[\widetilde{T}]$ such that
$X(m)=0$ and an $X\in[\widetilde{T}]$ such that $X(m)=1$.
\end{proof}

% \noindent{\bf Remark.}
Given the results of this section, the paths through this tree $T$ are
potential candidates for sets that are not density-1-bounding in the
generic degrees. However, we now show that all of the paths through
$T$ are density-1-bounding.

\begin{prop}\label{notenough}
  Let $T$ be the tree constructed in Proposition
  \ref{wiggletree}. Every $X\in[T]$ is density-1-bounding in the
  generic degrees.

\end{prop}

\begin{proof}
  Let $T$ be the tree constructed in Proposition \ref{wiggletree}, and
  for each $k\in\omega$, let $T_k$ be constructed as in the proof of
  Proposition \ref{wiggletree}.

  Let $X\in[T]$. Let $X_k=\{n:(2n+1)2^k\in X\}$, and note that by
  definition of $T$, we have that $X_k\in [T_k]$.

  Define $Y$ so that $Y_i=X_{2i+1}$. More formally, so that
  $(2n+1)2^i\in Y\leftrightarrow n\in X_{2i+1}$. Note that $Y$ is not
  generically computable by the same argument as why $X$ is neither
  generically nor coarsely computable. Note also that $Y$ is
  density-1, because for every $i$, every path through $T_{2i+1}$ is
  density-1. Finally, $Y\leq_\gen X$, because a generic oracle for $X$
  must contain density-1 many of the bits of $X_k$ for every $k$.
\end{proof}

The reason we are able to prove Proposition \ref{notenough} is that
the non-coarse computability requirements and non-generic
computability requirements are addressed independently in the proof of
Proposition \ref{wiggletree}. If there were a way of meeting both
sorts of requirements simultaneously, perhaps by meeting general
``non-dense-computability'' requirements, then this might shed
additional light on Question \ref{q1}.

\begin{question}
  If we define a real $A$ to be ``densely computable" if there is a
  partial computable $\phi$ such that $\{n:\phi(n)=A(n)\}$ is density
  1, then does there exist an infinite computable tree
  $T\subseteq 2^{<\omega}$ such that given any two paths in $[T]$, the
  two paths agree on density 1, and such that no path in $[T]$ is
  densely computable?

\end{question}

\section{Randoms and Generics}

In this section, we investigate the generic degrees of random reals,
and of generic reals. We show that both randomness and genericity
imply quasiminimality in the cofinite degrees, and therefore (by Lemma
\ref{qmpropagation}) in the mod-finite, generic, and coarse
degrees. In particular, this provides examples of quasiminimal sets
that are not density-1, potentially helping along the way to a
construction of a set that is not density-1-bounding. We also show
that both randomness and genericity imply density-1-bounding in the
generic degrees, potentially helping along the way to a proof that all
non-generically-computable sets are density-1-bounding.

\begin{prop}

  If $A$ is a weakly 1-random real, then $A$ is density-1-bounding in
  the generic degrees.
\end{prop}

\noindent{\bf Remark.}
Weak 1-randomness is implied by 1-randomness, and also by weak
1-genericity and therefore 1-genericity. Therefore, in particular,
this shows that 1-randoms and 1-generics are density-1-bounding in the
generic degrees.

%% added the first sentence cholak
\begin{proof}
  The proof we present is very similar to that of Theorem~2.2 of
  Hirschfeldt, Jockusch, McNicholl, and Schupp \cite{HJMS}.

  Define $B=\{\sigma\in2^{<\omega}:\sigma\nprec A\}$. For purposes of
  density, it is important to determine which coding of $2^{<\omega}$
  as a subset of $\omega$ is being used. We use the standard order on
  $2^{<\omega}$: first by length of $\sigma$, and then
  lexicographically among $\sigma$ of the same length, with one small
  adjustment. We start with the empty string being coded by $n=1$, not
  $n=0$, so that $\{\sigma:|\sigma|=i\}$ will be coded in to $P_i$, as
  in Definition \ref{pi}. (There is no string coded by $n=0$.)

  Note that for any $A$ at all, $B$ as defined above is density-1,
  because there is exactly one $\sigma$ of each length not in
  $B$. Likewise, for any $A$ at all, $A\geq_gB$, as follows.

  If $(A)$ is any partial oracle for $A$, consider the partial
  computation of $B$ given by enumerating all $\sigma$ that $(A)$ is
  able to rule out as potential initial segments of
  $A$:
  $$\phi^{(A)}(\sigma)=1\leftrightarrow\exists n\big(n\in \dom((A))\
  \&\ n<|\sigma|\ \&\ (A)(n)\neq\sigma(n)\big).$$
  % If $(A)$ is any generic oracle for $A$, then in particular $(A)$
  % has infinite domain. Consider the partial computation of $B$ given
  % by enumerating all $\sigma$ that $(A)$ is able to rule out as
  % potential initial segments of
  % $A$:
  % $$\phi^{(A)}(\sigma)=1\leftrightarrow\exists n\big(n\in \dom((A))\
  % \&\ n<|\sigma|\ \&\ (A)(n)\neq\sigma(n)\big).$$
  Note that for any partial oracle $(A)$, for $A$, $\phi^{(A)}$ is a
  partial computation of $B$.

  If $\dom((A))$ is infinite, then we claim that $\phi^{(A)}$ is a
  generic computation of $B$. To show this, we show that for each $e$,
  there are at most finitely many $i$ such that $\dom(\phi^{(A)})$ has
  a gap of size $2^{-e}$ at $P_i$, and then appeal to Lemma
  \ref{pilem}. To see this, let $e$ be given, and fix $i_0$ such that
  $|\dom((A))\upharpoonright i_0|\geq e+1$. For $i>i_0$, there are at
  most $2^{i-(e+1)}$ many $\sigma$ of length $i$ agreeing with $(A)$,
  and so $\phi^{(A)}$ enumerates at least $2^i-2^{i-(e+1)}$ many
  $\sigma$ of length $i$, and so $\dom(\phi^{(A)})$ does not have a
  gap of size $2^{-e}$ at $P_i$.

  % If $\dom((A))$ is infinite, then we claim that $\phi^{(A)}$ is a
  % generic computation of $B$. To show this, we appeal to Lemma
  % \ref{pilem}, showing that for each $e$, there are at most finitely
  % many $i$ such that $\dom(\phi^{(A)})$ has a gap of size $2^{-e}$
  % at $P_i$. To see

  It now remains to show that if $A$ is weakly 1-random, then the $B$
  that we built is not generically computable.

  Because $B$ is density-1, it is generically computable if and only
  if it has a density-1 c.e.\ subset. To prove that this cannot be the
  case, let $W_e$ be a c.e.\ set, thought of as coding a subset of
  $2^{<\omega}$. Consider
  $\mathcal{V}=\{X\in2^\omega:\forall\sigma\in W_e, \sigma\nprec
  X\}$.
  This is a $\Pi^0_1$ class, and we claim that if $W_e$ is density-1,
  then $\mathcal{V}$ is null, and we also claim that if $W_e$ is a
  subset of $B$, then $A$ is an element of $\mathcal{V}$.

  Both of these claims are straightforward from the definitions, and
  we leave the verification to the reader. If $A$ is weakly 1-random,
  then $A$ cannot be a member of any null $\Pi^0_1$ class, and so
  $W_e$ cannot be a generic computation of $B$.
  % For the first claim, note that by Lemma \ref{pilem}, if $W_e$ is
  % density-1, then for each $e$, there are at most finitely many $i$
  % such that $W_e$ has a gap of size $2^{-e}$ at $P_i$
\end{proof}

We now go on to prove that 1-generics and 1-randoms are quasiminimal
in the cofinite, (and hence mod-finite, generic and coarse) degrees.

% Both of these proofs will rely rather heavily on Observation
% \ref{notwrong}, which holds in the cofinite degrees by the same
% proof as in the generic degrees. In particular, in both cases, we
% will show that if $A$ is 1-generic or 1-random, if $\phi$ is a
% time-independent functional, and if, for every cofinite oracle
% $(A)$, for $A$, $\phi^{(A)}$ is total, then either $\phi^A$ is a
% computation of a computable set, or $\phi^A$ is multivalued.

Both of these proofs will be by the following lemma, which summarizes
and compiles the results of Sections 2 and 3 that we will use in this
section. Note that the results of Section 2 concerning
time-independent functionals apply equally well to cofinite reduction
as they do to generic reduction.

\begin{lem}\label{qmlemma}
  Assume $A$ is not quasiminimal in the cofinite degrees. Then there
  exists a time-independent Turing functional $\phi$ such that for any
  cofinite oracle $(A)$, for $A$, $\phi^{(A)}$ is total, and
  furthermore such that $\phi^A$ is not multivalued, and is a
  computation of a noncomputable real $B$.
\end{lem}

\begin{proof}
  If $A$ is not quasiminimal in the cofinite degrees, then by
  definition of quasiminimality, these is a noncomputable $B$ such
  that $\mathcal{R}(B)\leq_\cf A$. By Observation
  \ref{canusetimeindependent}, there is a time-independent Turing
  functional $\psi$ such that $\mathcal{R}(B)\leq_\cf A$ via
  $\phi$. Any cofinite oracle for $\mathcal{R}(B)$ can be uniformly
  used to compute $B$, and if we use the time-independent version of
  the Turing functional that computes $B$ from a cofinite oracle for
  $\mathcal{R}(B)$, and compose that functional with $\psi$, then we
  obtain a time-independent Turing functional $\psi$ such that for any
  cofinite oracle $(A)$, for $A$, $\phi^{(A)}$ is a total computation
  of $B$.

  In particular, because cofinite computations are never allowed to
  make mistakes, $\phi^A$ only produces correct outputs concerning
  $B$, and so is not multivalued.
\end{proof}

\begin{prop}\label{1g}

  If $A$ is weakly 1-generic, then $A$ is quasi-minimal in the cofinite
  degrees, and hence in the mod-finite, generic, and coarse degrees.

\end{prop}

\begin{proof}
  % Let $A$ be 1-generic, and assume $\mathcal{R}(B)\leq_\cf
  % A$.
  % Because of the fact that $B$ is uniformly computable from an
  % arbitrary cofinite oracle for $\mathcal{R}(B)$, we may assume
  % there is a time-independent Turing functional $\phi$ such that for
  % any cofinite oracle $(A)$, for $A$, $\phi^{(A)}$ is total, and is
  % a computation of $B$.
  Let $A$ be weakly 1-generic, and assume $A$ is not quasiminimal. By Lemma
  \ref{qmlemma}, there is a noncomputable $B$ and a time-independent
  Turing functional $\phi$ such that for any cofinite oracle $(A)$,
  for $A$, $\phi^{(A)}$ is total, and is a computation of $B$.  We
  prove then that $B$ is computable as follows.

  Consider $S=\{\sigma\in 2^{<\omega}:\phi^\sigma$ is
  multivalued$\}$. Note that $S$ is $\Sigma^0_1$ because
  $\sigma\in S\leftrightarrow\exists n\left(
  \phi^\sigma(n)\downarrow=1\ \&\ \phi^\sigma(n)\downarrow=0\right)$.
  If $S$ is dense, then because $A$ is weakly 1-generic, $A$ meets $S$, and
  so by Observation \ref{notwrong} it cannot be the case that $\phi^A$
  is a computation of $B$.

  If $S$ is not dense, then choose any $\tau$ that has no extensions
  in $S$. Let $m=|\tau|$. For any $X\in2^\omega$, let $X_m$ be the
  partial oracle for $X$ that does not halt on any of its first $m$
  bits. Then we claim that $B$ is computable by the functional $\psi$
  such that $\psi(n)$ searches for an $X$ such that
  $\phi^{X_m}(n)\downarrow$ and then outputs the same value as the
  found output.

  This computation is total because for any cofinite oracle $(A)$, for
  $A$, $\phi^{(A)}$ is total, and so in particular, for any $n$,
  $\phi^{(A)_m}(n)\downarrow$. Furthermore, when this computation
  halts, it gives a correct output for the following reason. Assume
  not, and fix $X$ such that $\phi^{X_m}(n)\downarrow\neq B(n)$. Let
  $k$ be the use of this computation (the smallest number such that
  only $X\upharpoonright k$ was required for the computation of
  $\phi^{X_m}(n)$). Note then that by hypothesis
  $\phi^{A_k}(n)\downarrow =B(n)$. Let $l$ be the use of this
  computation.

  Let $\sigma\in 2^{<\omega}$ be given by
  $\sigma\upharpoonright m=\tau$,
  $\sigma\upharpoonright [m,k)=X\upharpoonright [m,k)$, and
  $\sigma\upharpoonright [k,l)=A\upharpoonright [k,l)$. Note then that
  $\sigma\succeq \tau$, and furthermore, because $\phi$ is a
  time-independent functional, $\phi^\sigma(n)=\phi^{X_m}(n)\neq B$,
  and also $\phi^\sigma(n)=\phi^{A_k}(n)= B$, so $\phi^\sigma$ is
  multivalued, contradicting our choice of $\tau$.
\end{proof}

\begin{prop}\label{1r}

  If $A$ is 1-random, then $A$ is quasi-minimal in the cofinite
  degrees, and hence in the mod-finite, generic, and coarse degrees.

\end{prop}

The proof of Proposition \ref{1r} is somewhat more involved than the
proof of Proposition \ref{1g}. It will use Lemma \ref{qmlemma} as well
as a fairly subtle control of the halting measure of
$\phi$. Throughout the remainder of the section, we will always assume
that $\phi$ is a time-independent Turing functional.

\begin{defn}
  Fix $\phi$, and an integer $n$. Let $k\leq l$ be integers.

  For any real $X$, let $X_k$ be the partial oracle for $X$ that does
  not halt on inputs less than $k$, and let $X_{k,l}$ be the partial
  oracle for $X$ that does not halt on inputs less than $k$ or greater
  than or equal to $l$.

  Define $\mu_k=\mu(\{X:\phi^{X_k}(n)\downarrow\})$, the halting
  measure of $\phi$-computations that do not use the first $k$ bits of
  their oracles.

  Similarly, define
  $\mu_{k,l}=\mu(\{X:\phi^{X_{k,l}}(n)\downarrow\})$.

\end{defn}

Note that $\phi$ and $n$ are suppressed in the notation for brevity.

\begin{obs}\label{mu}
  \

  If $k_0<k_1$, then for any $l$, $\mu_{k_0,l}\geq\mu_{k_1,l}$, and
  also $\mu_{k_0}\geq\mu_{k_1}$.

  If $l_0<l_1$, then for any $k$, $\mu_{k,l_0}\leq\mu_{k,l_1}$.

  For any $k$, $\mu_k=\lim_{l\rightarrow\infty}\mu_{k,l}$.

\end{obs}

\begin{proof}
  The first two facts follow from Observation \ref{moreismore}, which
  says that larger oracles never have smaller halting sets. The third
  fact is because any computation that halts does so using only
  finitely much from its oracle.
\end{proof}

%%%%%%%%%
% \begin{obs}\label{notwrong}
%   Let $A$ and $B$ be reals. Assume that $A\leq_{\gen}B$ via
%   $\phi$. Then for any $n,x$, and for any partial oracle $(B)$ for
%   $B$, if $\phi^{(B)}(n)=x$, then $A(n)=x$.

%   In particular, given any $X$ and $Y$, if $\phi^Y$ is multi-valued,
%   then it cannot be the case that $X\leq_{\gen}Y$ via $\phi$.  In
%   particular, given any $X$ and $Y$, and any partial oracle $(Y)$
%   for $Y$, if $\phi^Y$ is multi-valued, then it cannot be the case
%   that $X\leq_{\gen}Y$ via $\phi$.
% \end{obs}

% \begin{obs}\label{canusetimeindependent}
%   Let $A$ and $B$ be reals. Assume that $A\leq_{\gen}B$ via
%   $\phi$. Then $A\leq_{\gen}B$ via the time-independent version of
%   $\phi$.
% \end{obs}

%%%%%%%%%

\begin{lem}\label{muis1}
  Let $A$ be 1-random, $\phi$ a time-independent Turing functional,
  and $n$ an integer. If for every cofinite oracle $(A)$ for $A$,
  $\phi^{(A)}(n)\downarrow$, then, for that $\phi$ and $n$, for every
  $k$, $\mu_k=1$.
\end{lem}

\begin{proof}
  Assume not. Fix $k$ such that $\mu_k\neq 1$, and fix some computable
  $\epsilon>0$ such that $\mu_k<1-\epsilon$.

  Consider the open sets $U_i$ defined by $X\in U_i$ if $\phi^{X_k}$
  converges for at least $i$ different independent reasons. We make
  this precise by an inductive definition as follows.

  Let $U_0=2^\omega$, and for every $X$, ${k_{X,0}}=k$.

  Having defined $U_i$, and ${k_{X,i}}$ for every $X$ in $U_i$, we let
  $U_{i+1}=\{X\in U_i:\phi^{X_{k_{X,i}}}\downarrow\}$, and for every
  $X$ in $U_{i+1}$, we let $k_{X,i+1}$ be the minimum number $s$ such
  that $\phi^{X_{k_{X,i}}}\downarrow$ in less than $s$ steps, and
  using less than the first $s$ bits of the oracle $X$. Such an $s$
  can be found computably in $X$ for any $X$ in $U_{i+1}$, and so each
  $U_i$ is defined $\Sigma_1$ in the previous one. Thus, the $U_i$ are
  uniformly $\Sigma_1$ sets.

  Furthermore, $\mu(U_i)\leq(1-\epsilon)^i$. This is because to
  determine whether or not $X\in U_i$, we must see that $\phi^{X_k}$
  halts (a measure $<1-\epsilon$ event), then we ignore everything
  that caused $\phi^{X_k}$ to halt, and require that
  $\phi^{X_{k_{X,0}}}$ halts (another measure $<1-\epsilon$ event) and
  so on. Each event is independent, and so the probability of meeting
  all of them is equal to their product.

  If for every $l$, $\phi^{A_l}\downarrow$, then $A$ must be in every
  $U_i$. This contradicts the assumption that $A$ is 1-random.
\end{proof}

This lemma will be what we require in order to construct a collection
of ``towers'' that will prove our contradiction. We will use the ideas
of a ``90\%-halting tower,'' an ``80\%-agreement tower,'' and a
``60\%-disagreement tower.'' The numbers 90\%, 80\%, and 60\% are not
special: the only important facts about them are that
$100\%>90\%>80\%>50\%$, and also that $80\%>60\%>0\%$. For our proof,
it will be convenient that $0.6<0.8^2$.

\begin{defn}
  Fix $\phi$ and $n$, and $k$. Then for that $\phi$, $n$ and $k$, a
  \emph{90\%-halting tower starting at $k$} is a sequence of numbers
  $\langle k_i:i\in\omega\rangle$ such that $k_0=k$, for every $i$,
  $k_{i+1}>k_i$, and $\mu_{k_i,\mu_{k_{i+1}}}>0.9$.
\end{defn}

\begin{obs}
  Fix $\phi, n$.

  If there exists some $k$ such that there is a 90\%-halting tower
  starting at $k$, then for every $k$, there is a 90\%-halting tower
  starting at $k$. Moreover, that 90\% halting tower can be found
  uniformly computably in $n, k$.

\end{obs}

\begin{proof}
  The obvious greedy algorithm works for this. Let $k_0=k$, and search
  for a $k_1$ such that $\mu_{k_0,\mu_{k_{1}}}>0.9$. Eventually such a
  $k_1$ will be found because there is a 90\%-halting tower,
  $\langle l_i:i\in\omega\rangle$ starting somewhere, and that other
  halting tower must have some $i$ such that $l_i\geq k_0$. But then
  $\mu_{k_0,l_{i+1}}\geq\mu_{l_i,l_{i+1}}\geq0.9$ (using Observation
  \ref{mu}). Thus, $l_{i+1}$ would work as $k_1$, although some other
  smaller or larger $k$ might be found first. We then proceed
  inductively to define each $k_i$.
\end{proof}

\begin{lem}\label{90towers}
  Assume that $\phi$ is such that for every $n$, and for every $k$,
  $\mu_k=1$. Then, for that $\phi$ and for every $n$ and $k$, there
  exists a 90\%-halting tower starting at $k$.

\end{lem}

\begin{proof}
  Once again, the obvious greedy algorithm works. We let $k_0=k$, and
  for each $i$, let $k_{i+1}$ be the first $l$ found such that
  $\mu_{k_i,l}\geq0.9$. Such an $l$ exists because
  $\lim_{l\rightarrow\infty}\mu_{k_i}=1$.
\end{proof}

\begin{defn}
  Fix $\phi, n, k$, and $\langle k_i:i\in\omega\rangle$, a
  90\%-halting tower starting at $k$.  Then
  $\langle k_i:i\in\omega\rangle$ is an \emph{80\%-agreement tower} if
  there exists some $v\in\{0,1\}$ such that:

  \begin{enumerate}
    % \item For all but finitely many $i$,
    %   $\mu\{X:\phi^{X_{k_i,k_{i+1}}}(n)=v\}>0.8$.

  \item There exists an $i$ such that
    $\mu\{X:\phi^{X_{k_i,k_{i+1}}}(n)=v\}>0.8$ and,

  \item There does not exist an $i$ such that
    $\mu\{X:\phi^{X_{k_i,k_{i+1}}}(n)\neq v\}>0.8$.
  \end{enumerate}
  In this case, we will sometimes say that
  $\langle k_i:i\in\omega\rangle$ is an 80\%-agreement tower with
  value $v$.
\end{defn}

The purpose of these 80\% agreement towers is that they will allow us
to compute $B$ without needing $A$ as an oracle. We will later prove
that they must exist, but first we show how they can be used to
compute $B$.

\begin{lem}\label{80isright}
  % Assume that $A$ is 1-random, and that and that there is some $B$
  % such that for every $l$, $\phi^{A_l}$ is a computation of $B$. Fix
  % $n, k, v$. Assume that for that $\phi, n$, there exists a
  % computable 80\%-agreement tower, $\langle k_i:i\in\omega\rangle$,
  % with value $v$ starting at $k$. Then $B(n)=v$.
  Assume that $A$ is 1-random, and that $\phi^{A}$ is a computation of
  $B$. Fix $n, k, v$. Assume that for that $\phi, n$, there exists a
  computable 80\%-agreement tower, $\langle k_i:i\in\omega\rangle$,
  with value $v$ starting at $k$. Then $B(n)=v$.
\end{lem}

\begin{proof}

  Assume the hypotheses are true and the conclusion is false.

  % By symmetry, assume that $v=0$, and that $B(n)=1$.
  Consider the open sets $U_i=\{X:\phi^{X_{k_i,k_{i+1}}}(n)=v\}$. We
  have that $A$ is not in any $U_i$ because $\phi^{A}(n)=B(n)\neq
  v$.
  Also, for every $i$,
  $\mu\{X:\phi^{X_{k_i,k_{i+1}}}(n)\downarrow\}>0.9$, and there does
  not exist an $i$ such that
  $\mu\{X:\phi^{X_{k_i,k_{i+1}}}(n)\neq v\}>0.8$. Therefore
  $\mu(U_i>10\%)$.
  % $U_j=\{X:(\forall i\leq j)\phi^{X_{k_i,k_{i+1}}}(n)=v$.

  Let $C_i$ be the complement of $U_i$, and let $C=\bigcap_i C_i$. The
  measure of the intersection of the $C_i$ is equal to the product of
  the measures of the $C_i$ because each $C_i$ is defined in terms of
  only the bits of $X$ between $k_i$ and $k_{i+1}$. Therefore $C$ is a
  null $\Pi^0_1$ set, and $A\in C$. This contradicts the assumption
  that $A$ is 1-random.
\end{proof}

To prove that they must exist, we will use 60\%-disagreement towers,
which will derandomize $X$ if $\phi$ does not produce enough 80\%
agreement towers. These 60\%-disagreement towers, unlike most of our
other work in this proof, will not be fixed to a specific $n$.

\begin{defn}

  Let $\phi$ be given. Then a \emph{60\%-disagreement tower} for
  $\phi$ is a sequence $\langle m_i:i\in\omega\rangle$ such that for
  every $i$,

  $\mu\left\{X:(\exists n)\left(\phi^{X_{m_i,m_{i+1}}}\text{ is a
        multivalued function on }n\right)\right\}>0.6$.

\end{defn}

\begin{obs}
  If there is a 60\%-disagreement tower for $\phi$, then there is a
  computable 60\%-disagreement tower for $\phi$.

\end{obs}

\begin{proof}
  Again, a greedy algorithm produces a a computable 60\%-disagreement
  tower.
\end{proof}

\begin{lem}\label{80or60}
  Let $\phi$ be be given.
  % Let $\phi$ be such that for every $k$, there is some $n$ such that
  % there exists 90\%-halting tower for $n$ starting at $k$.
  Then, either there exists a $k$ such that every 90\%-halting tower
  (for any $n$) starting at $k$ is an 80\%-agreement tower, or there
  exists a 60\%-disagreement tower for $\phi$.
  % Then, either there exists an $l$ such that for all $k>l$, every
  % 90\%-halting tower (for any $n$) starting at $k$ is an
  % 80\%-agreement tower, or there exists a 60\%-disagreement tower
  % for $\phi$.

\end{lem}

\begin{proof}
  % Assume the hypothesis, and
  Assume the first clause is false. We construct a 60\%-disagreement
  tower as follows.

  Let $m_0=0$. Choose some $n$ and some
  $\langle k_i:i\in\omega\rangle$ that is a 90\%-halting tower for
  that $n$ starting at $k=m_0$ that is not an 80\%-agreement tower.

  \vspace{5pt}

  \noindent{\bf Case 1:} If it is not an 80\%-agreement because it
  never reaches a consensus (ie, it fails the first clause in the
  definition of 80\%-agreement towers), then choose $j$ to be large
  enough that $(0.9)(1-(0.9)^{j-1})>(0.6)$, and let $m_1=k_j$.
  % \noindent{\bf Case 1:} If it is not an 80\%-agreement because it
  % never reaches a consensus (ie, it fails the first clause in the
  % definition of 80\%-agreement towers), then at every level of the
  % tower, at least 10\% many $X$ give the opposite output from the
  % rest of the $X$. Thus, in particular, for any fixed value of $v$,
  % at every level, at least 10\% of the $X$ disagree with that value
  % of $v$. So now, choose $j$ to be large enough that
  % $(0.9)(1-(0.9)^{j-1})>(0.6)$, and let $m_1=k_j$.
  We claim then that
$$\mu\left\{X:\left(\phi^{X_{m_0,m_{1}}}\text{ is a multivalued function on }n\right)\right\}>0.6.$$
This is because of the following calculation.

Let $S$ be the set of $X$ such that $\phi^{X_{k_0,k_1}}(n)$
halts. Note that $\mu(S)\geq 0.9$, by definition of a 90\% halting
tower.
% For each $X\in S$, let the $i$th output of $X$ be the value, if it
% exists, of $\phi^{X_{k_i,k_{i+1}}}(n)$. (If
% $\phi^{X_{k_i,k_{i+1}}}(n)$ is multivalued, then $X$ has more than
% one $i$th output.) Note that for all $X\in S$, $X$ has at least one
% 0th
% output. %Note also that for $i_0\neq i_1$, the $i_0$th output of $X$ and the $i_1$th output of $X$ are completely independent, as they use disjoint portions of the oracle $X$.

We show that given any $i>0$, and any $\sigma$ with $|\sigma|=k_i$, if
$\sigma\prec X$ for some $X\in S$, and $\phi^\sigma(n)$ is not
multivalued, then $\phi^\tau(n)$ is multivalued for at least $10\%$ of
all $\tau\succ\sigma$ with $|\tau|=k_{i+1}$. This will show that,
$\phi^\tau(n)$ is multivalued for at least measure
$(0.9)(1-(0.9)^{j-1})$ many $\tau$ of length $k_j$ because each time
we increase $i$, $10\%$ of all strings that have not yet become
multivalued become multivalued, and we started with measure $0.9$ many
strings.

Let $\sigma$ be as above, and note that because $\sigma\prec X$ for
some $X\in S$, we have that $\phi^\sigma(n)$ is defined. By hypothesis
of Case 1, at least 10\% many $X$ give the opposite output from the
rest of the $X$ when computing $\phi^{X_{i,i+1}}$. In particular, this
implies that at least 10\% many $X$ halt and give the opposite output
from $\phi^\sigma(n)$ when computing $\phi^{X_{i,i+1}}$. This
computation does not use any bits of $X$ less than $|\sigma|$, or
greater than $k_{i+1}$, and so 10\% many of all $\tau\succ\sigma$ with
$|\tau|=k_{i+1}$ cause $\phi^\tau(n)$ to halt and give the opposite
output from $\phi^\sigma(n)$. But because $\tau\succ\sigma$,
$\phi^\tau(n)$ also halts and gives the same output as
$\phi^\sigma(n)$, and so $\phi^\tau(n)$ is multivalued.

% Let $S$ be the set of $X$ such that $\phi^{X_{k_0,k_1}}$ halts. Note
% that $\mu(S)\geq 0.9$, by definition of a 90\% halting tower. For
% each $X\in S$, let the $i$th output of $X$ be the value, if it
% exists, of $\phi^{X_{k_i,k_{i+1}}}$. If $\phi^{X_{k_0,k_1}}$ is
% multivalued, let the 0th output of $X$ be 0. Note that for all
% $X\in S$, $X$ has a 0th output. For $i\neq 0$, if
% $\phi^{X_{k_i,k_{i+1}}}$ is multivalued, then let the $i$th output
% of $X$ be an output that disagrees with the 0th output of $X$.

% The measure of the set of $X$ that halt with the first $k_1$-many
% bits of their oracles is 0.9. For each of those $X$, call the value
% of that output their ``first output.'' As we increase the size of
% the oracle from $k_i$ to $k_{i+1}$, at least another 10\% of the $X$
% that have not yet given an output that disagrees with their first
% output give an output that disagrees with their first output. Thus,
% at each level, the set of $X$ that have not yet disagreed with one
% of their outputs gets smaller by 10\%, i.e. gets multiplied by
% 0.9. Eventually, an arbitrarily large fraction of the original
% measure $0.9$-many oracles disagree with their original outputs, and
% so in particular, there is a stage when measure 0.6-many oracles
% have disagreed with themselves.

\vspace{5pt}

\noindent{\bf Case 2:} If the 90\%-halting tower is not an
80\%-agreement tower because it does reach a consensus, but it also
reaches the opposite consensus at some point, then let $j$ be large
enough that both kinds of consensus get reached before $k_j$, and let
$m_1=k_j$. Then at least 80\% of the oracles from the first consensus
arrive at the opposite conclusion with their later information, and
so, in particular, they give a multivalued function if enough of
the oracle is taken to witness both of those computations. We have
that $0.8^2=0.64>0.6$, and so at that stage, at least measure
$0.6$-many oracles produce multivalued functions.

\vspace{5pt}

We then repeat the construction, choosing a potentially new $n$, and a
new 90\%-halting tower beginning at $k=m_1$ to find $m_2$, and then we
repeat with $m_2$, $m_3$ etc.  At each stage, this ensures that for
$60\%$ of all $X$, $\phi^{X_{m_i,m_{i+1}}}(n)$ is multivalued for some
value of $n$.
% We then repeat the construction with a tower beginning at $m_1$, and
% then $m_2$, etc.
\end{proof}

\begin{lem}\label{no60}
  Assume that $A$ is 1-random, and that there is some $B$ such that
  for every $k$, $\phi^{A_k}$ is a computation of $B$. Then there is
  no 60\%-disagreement tower for $\phi$.
\end{lem}

\begin{proof}
  If it is true that for every $k$, $\phi^{A_k}$ is a computation of
  $B$, then in particular, there is no partial oracle for $A$ that
  causes $\phi$ to be a multi-valued function. Thus, in particular,
  there can be no $m_i,m_{i+1}$ such that $A_{m_i,m_{i+1}}$ is in the
  ``60\% disagreement part'' of the 60\%-disagreement tower.

  Thus, if there existed a 60\%-disagreement tower for $\phi$, then
  $A$ would be in the intersection, over all $i$, of the reals $X$
  such that $\phi^{X_{m_i,m_{i+1}}}$ is not a multivalued function on
  $n$. This is an intersection of uniformly $\Pi^0_1$ sets, and so it
  is a $\Pi^0_1$ set. Furthermore, it is a measure 0 set because each
  one of the sets was at most measure 0.4, and the sets each use
  disjoint parts of the oracle, so the measure of their intersection
  is the product of their measures. A 1-random real $A$ cannot be in a
  $\Pi^0_1$ null set, and so there cannot be a 60\%-disagreement tower
  for $\phi$.
\end{proof}

We are now ready to prove Proposition \ref{1r}, which we restate here for clarity.\\

\noindent{\bf Proposition \ref{1r}.} \emph{If $A$ is 1-random, then
  $\widetilde{\mathcal{R}}(A)$ (and hence $A$) is quasi-minimal in the
  generic degrees.}

\begin{proof}
  Assume $A$ is 1-random. We show that $A$ is quasi-minimal in the
  cofinite degrees, and hence that $\widetilde{\mathcal{R}}(A)$ is
  quasi-minimal in the generic degrees. To prove this, assume that for
  every cofinite oracle $(A)$ for $A$, $\phi^{(A)}$ is a computation
  of $B$. We must prove that $B$ is therefore computable.

  By Lemma \ref{muis1}, for that $\phi$ and for every $n$, and every
  $k$, we must have that $\mu_k=1$. By Lemma \ref{90towers}, we
  therefore have that for every $n$ and $k$, there is a 90\%-halting
  tower starting at $k$. By Lemmas \ref{80or60} and \ref{no60}, we
  have that there exists an $l$ such that every
  90\%-halting tower (for any $n$) starting at $l$ is an
  80\%-agreement tower.

  Fix such an $l$. We now compute $B$ through a ``majority vote''
  trick. To compute $B(n)$, wait until 80\% of all $X$ have the
  property that $\phi^{X_l}(n)\downarrow$, giving the same output,
  then halt and give that output. We must now verify that this will
  happen at some point, and that when it happens, it gives the correct
  output.

  We know this will eventually happen, because there is a 90\%-halting
  tower starting at $l$, and it is an 80\% agreement tower. Thus there
  exist some $v,k_i, k_{i+1}$ such that $l\leq k_i\leq k_{i+1}$, and
  $\mu\{X:\phi^{X_{k_i,k_{i+1}}}(n)=v\}>0.8$. However, for any $X$, if
  $\phi^{X_{k_i,k_{i+1}}}(n)=v$, then $\phi^{X_{l}}(n)=v$, because
  $X_l$ is is a partial oracle extension of $X_{k_i,k_{i+1}}$, and so
  we have that $\mu\{X:\phi^{X_{l}}(n)=v\}>0.8$.

  Furthermore, the $v$ that we find must be the value $v$ of some 80\%
  agreement tower. This is because we may build a 90\% halting tower
  for which $k_0=l$, and $k_1$ is large enough to witness that measure
  80\% many $X$ give output $v$. By assumption, this tower is an 80\%
  agreement tower, and by definition of 80\% agreement tower, no
  ``floor'' of the tower can have 80\% many $X$ give an output other
  than the value $v$ of that tower, and so the $v$ that we found is
  the $v$ of that tower.

  Therefore, by Lemma \ref{80isright}, that $v$ must be $B(n)$, and so
  we have correctly computed $B(n)$ without using $A$ as an oracle.
\end{proof}

\section{Nonuniform generic reducibility}

In this section, we consider nonuniform generic reducibility, which
has the property that the functional $\phi$ is allowed to change
depending on the generic oracle for $A$. This section is joint work
with Denis Hirschfeldt.

\begin{defn}
  \normalfont Let $A$ and $B$ be reals. Then \emph{$B$ is
    non-uniformly generically reducible to $A$} if for every generic
  oracle, $(A)$, for $A$, there exists a Turing functional $\phi$ such
  that $\phi^{(A)}$ is a generic computation of $B$. In this case, we
  write $B\leq_{\ngen}A$.
\end{defn}

Note that the Turing degrees embed into the nonuniform generic degrees
by the same map as the one used to embed them into the uniform generic
degrees, and so we may again define a degree to be quasiminimal if it
is not above any nonzero embedded Turing degrees.  We show weakly
2-randoms are quasiminimal in the nonuniform generic degrees, and also
that 1-generics are quasiminimal in the nonuniform generic
degrees. The inspiration for this section comes from a preprint of
\cite{HJKS}, and the realization that much of their work was
complementary to the work in this paper.
% We show weakly 2-randoms are quasiminimal in the nonuniform generic
% degrees. The inspiration for this section comes from a preprint of
% \cite{HJKS}, and the realization that much of their work was
% complementary to the work in this paper.

In a preprint of \cite{HJKS}, Hirschfeldt, Jockusch, Kuyper and Schupp
proved that $\Delta^0_2$ 1-randoms are not quasiminimal in the
nonuniform coarse degrees (Corollary 3.11 from \cite{HJKS}), but that
weakly 2-randoms are quasiminimal in the nonuniform coarse degrees
(Corollary 3.3 from \cite{HJKS}). They also asked whether 1-randoms
are quasiminimal in the uniform coarse degrees.

Proposition \ref{1r} answers this, showing that 1-randoms are
quasiminimal in both the uniform coarse and uniform generic
degrees. In the published version of \cite{HJKS}, the authors also
modify their proof of Corollary 3.11 to show that $\Delta^0_2$
1-randoms are not quasiminimal in the nonuniform generic degrees. In
this section, we combine our proof of Proposition \ref{1r} with the
proof of Corollary 3.3 from \cite{HJKS} to prove that weakly 2-randoms
are quasiminimal in the nonuniform generic degrees.

We use the following result, implicitly proved in the proof of Theorem
3.2 of \cite{HJKS}, but stated here in the form that we will use.

\begin{lem}[Hirschfeldt, Jockusch, Kuyper, Schupp
  \cite{HJKS}]\label{HJKSlem}

  Assume $A$ is weakly 2-random, $B$ is noncomputable, and $k>1$. For
  each $i<k$, let $A_{=i}=\{n:kn+i\in A\}$, and let
  $A_{\neq i}=\bigoplus_{j\neq i} A_{=j}$.

  Then $(\exists i<k)(B\nleq_T A_{\neq i})$.

  Furthermore, for every $i$, $A_{=i}$ is 1-random relative to
  $A_{\neq i}$.
\end{lem}

\begin{proof}[Proof (sketch)]
  Assume that for every $i$, $B\leq_T A_{\neq i}$. By a generalized
  form of Van Lambalgen's Theorem \cite{vL}, we have that for every
  $i$, $A_{=i}$ is 1-random relative to $A_{\neq i}$, and so therefore
  relative to $B\oplus A_{\neq i}\equiv_T A_{\neq i}$. By the same
  generalized form of Van Lambalgen's Theorem relativized to $B$, we
  therefore have that $A$ is 1-random relative to $B$.  We also have
  that $B\leq_TA$ (because $B\leq_T A_{\neq i}$), and so we can
  conclude that $B$ is a base for 1-randomness, and hence is
  $K$-trivial \cite{HNS}.

  A weakly 2-random cannot compute any noncomputable $\Delta^0_2$ sets
  \cite{DNWY}, and so cannot compute any noncomputable $K$-trivials.
\end{proof}

\begin{thm}[Cholak, Hirschfeldt, Igusa]\label{nu1r}
  Assume $A$ is weakly 2-random. Then $A$ is quasiminimal in the
  nonuniform generic degrees.

\end{thm}

The proof will use a relativized version of Proposition \ref{1r}. The
proof of Proposition \ref{1r} made ample use of uniformity, so we
begin our argument with a forcing argument that will allow us to
reduce the question to a uniform question. The uniform question will
then be answered using Lemma \ref{HJKSlem} and Proposition \ref{1r}.

Note also that in in Section 2 we were working with uniform generic
reducibility, and so we do not have access to Observation
\ref{canusetimeindependent}, which said that we may use
time-independent functionals, and hence ignore the time dependence in
our partial oracles. Because of this we will work directly with
time-dependent partial oracles. In subsequent work Astor, Hirshfeldt,
and Jockusch generalize the proof of Theorem \ref{nu1r} to prove a
full analogue of Observation \ref{canusetimeindependent}, which would
simplify many of the steps of our proof.

% When we proved this result we do not know whether we may, without
% loss of generality, use time-independent functionals rather than
% ordinary Turing functionals.  After a first version of this paper
% appeared Hirschfeldt proved that we can replace ordinary Turing
% functionals with time-independent Turing functionals.  That work
% will be forthcoming. But since using time-independent functionals
% does not significantly shorten our proof we will be using ordinary
% Turing functionals.

% Note also that we do not know whether we may, without loss of
% generality, assume we are using time-independent functionals, so we
% will be working with ordinary Turing functionals. (It is an open
% question whether time-independent nonuniform generic reducibility is
% equivalent to time-dependent nonuniform generic reducibility. See
% \cite{I1}.)

We remind the reader that a partial oracle is coded as a set of
ordered triples $\langle n,x,l\rangle$, with $n$ as the input, $x$ as
the output, and $l$ as the number of steps required for the oracle to
halt. As such, we will use the following nonuniform version of Lemma
\ref{qmlemma}.

\begin{lem}\label{qmlemma2}
  Assume $A$ is not quasiminimal in the nonuniform generic degrees
  degrees. Then there exists a noncomputable real $B$ such that for
  any generic oracle $(A)$, for $A$, $B\leq_T(A)$.
  % there is a Turing functional $\phi$ such that $\phi^{(A)}$ is
  % total and is a computation of $B$.

\end{lem}

\begin{proof}[Proof (Lemma \ref{qmlemma2})]
  By definition of quasiminimality, if $A$ is not quasiminimal, then
  there is a noncomputable $B$ such that
  $\widetilde{\mathcal{R}}(\mathcal{R}(B))\leq_{\ngen}A$. So
  $\widetilde{\mathcal{R}}(\mathcal{R}(B))$ is generically computable
  from every generic oracle for $A$. But every generic oracle for
  $\widetilde{\mathcal{R}}(\mathcal{R}(B))$ can compute $B$, and so
  $B$ is computable from every generic oracle for $A$.
\end{proof}

\begin{proof}[Proof (Theorem \ref{nu1r})]
  First, we construct a forcing poset that will allow us to define a
  generic generic oracle, $G$ for $A$. The poset $\mathcal{P}$ will
  consist of finite approximations to partial oracles for $A$,
  together with a restriction saying that in the future, the partial
  oracles will need to have domain at least a certain size. In order
  to ensure that generic generic oracles are not total, $p$ will
  determine the entire behavior of $G\upharpoonright m$ for some $m$,
  so in particular, extensions of $p$ will not be allowed to halt at
  locations smaller than $m$.

  A \emph{finite partial oracle} $\sigma$ for $A$ is given by a number
  $m=|\sigma|$ and a subset of $m\times 2\times\omega$ which, thought
  of as a subset of $\omega\times 2\times\omega$ would be a partial
  oracle for $A$. For $n<m$, $\sigma\upharpoonright n$ is a shorthand
  for the partial oracle $\tau$ such that $|\tau|=n$ and for $k<n$,
  $\langle k,x,l\rangle\in\tau\leftrightarrow \langle
  k,x,l\rangle\in\sigma$.
  As with other partial oracles,
  $\dom(\sigma)=\{n<|\sigma|: \exists x\exists l\langle
  n,x,l\rangle\in\sigma\}$.

  We define the poset $\mathcal{P}$ to be the set of ordered pairs
  $\langle\sigma,\epsilon\rangle$ such that $\sigma$ is a partial
  oracle for $A$, $\epsilon>0$, and
  $\frac{\dom(\sigma)}{|\sigma|}> 1-\epsilon$. Given conditions
  $p=\langle\sigma,\epsilon\rangle$, and
  $q=\langle\tau,\delta\rangle$, we say that $q\leq p$ if
  $|\tau|\geq |\sigma|$, %% added cholak
  $ \tau \restriction |\sigma| = \sigma$
%% end added
$\delta\leq\epsilon$, and for all $n$, if
  $|\sigma|\leq n\leq|\tau|$, then
  $\frac{\dom(\tau\upharpoonright n)}{n}>1-\epsilon$.

  A generic generic oracle $G$ for $A$ is given by taking a
  sufficiently generic filter $\widetilde{G}$ for $\mathcal{P}$ and
  letting
  $G=\bigcup_{\langle \sigma,\epsilon\rangle\in
    \widetilde{G}}\sigma$.
  Note that it is dense to decrease $\epsilon$ below any positive
  number, and so a generic generic oracle for $A$ is a generic oracle
  for $A$.

  In this proof, we only use genericity of $\widetilde{G}$ for two purposes: ensuring that $G$ is a generic oracle, and ensuring that, given an arbitrary $\phi$, if $\phi^G$ is a computation of $B$, then there is a condition $p\in\widetilde{G}$ that forces that $\phi^G$ is a computation of $B$. We do not wish to explicitly count quantifiers, but $\mathcal{P}$ is $A$-computable, so some small level of genericity relative to $A\oplus B$ is sufficient.\\

  So, let $A$ be weakly 2-random, and assume that $A$ is not
  quasiminimal. By Lemma \ref{qmlemma2}, fix $B$ noncomputable such
  that $B\leq_T(A)$ for every generic oracle $(A)$ for $A$. Let $G$ be
  a generic generic oracle for $A$, and fix $\phi$ such that $\phi^G$
  is a computation of $B$. Fix a condition
  $p=\langle\sigma,\epsilon\rangle$ that forces that $\phi^G$ is a
  computation of $B$.  Fix $k$ such that $\frac1k<\epsilon$. By Lemma
  \ref{HJKSlem}, fix $i<k$ such that $B\nleq_T A_{\neq i}$.

  We then claim that relativized to $A_{\neq i}$, $B$ is uniformly
  computable from an arbitrary cofinite oracle for $A_{=i}$, and also
  that, relative to $A_{\neq i}$, $A_{=i}$ is 1-random, contradicting
  Proposition \ref{1r} relativized to $A_{\neq i}$.

  \vspace{5pt}

  \noindent Proof of Claim: Let $X$ be an arbitrary cofinite oracle
  for $A_{=i}$. Let $\mathcal{F}(X)$ be the cofinite oracle for $A$
  defined as follows.

  For $m\geq |\sigma|$, let
  $S_m=\{n:(n\in\dom(\sigma))\vee(|\sigma|\leq n< m)\vee(n\geq m\ \&\
  n\not\equiv i \mod k)\}$.
  Choose $m_0$ sufficiently large that for all $m\geq \dom(\sigma)$,
  $\frac{|S_{m_0}\upharpoonright m|}{m} >1-\epsilon$. (Such an $m_0$
  exists because $k$ was chosen so that $\frac1k<\epsilon$.)

  Let $\mathcal{F}(X)$ be the cofinite oracle for $A$ that agrees with
  $\sigma$ on $|\sigma|$, that halts immediately on all $m$ between
  $|\sigma|$ and $m_0$, that halts immediately on all $m\geq m_0$ if
  $m\not\equiv i \mod k$, and so that if $m>m_0$ and
  $m\equiv i\mod k$, then $\mathcal{F}(X)(m)=X(\frac{m-i}{k})$
  (halting if and only if $X$ halts, giving the same output if it does
  halt, and halting with the same $l$ value).

  Finally, we define $\psi$ so that $\psi^{X\oplus A_{\neq i}}(n)$
  searches for a partial oracle $Y$ such that:
  \begin{itemize}

  \item $Y\upharpoonright |\sigma|=\sigma$

  \item $\dom(Y)\subseteq\dom(\mathcal{F}(X))$

  \item for each $m\in\dom(Y)$, $Y(m)=\mathcal{F}(X)(m)$

    % \item for each $n$, there exists at most one $l$ such that
    %   $\langle n,X(n),l\rangle\in Y$

  \item $\phi^Y(n)\downarrow$

\end{itemize}
and when it finds such a $Y$, then
$\psi^{X\oplus A_{\neq i}}(n)\downarrow=\phi^Y(n)$.

(In essence, $\psi^{X\oplus A_{\neq i}}$ is almost a time-independent
version of $\phi^{\mathcal{F}(X)}$, except that it restricts its
attention only to $Y$'s that look potentially like extensions of $p$.)

\vspace{5pt}

It remains to show that $\psi^{X\oplus A_{\neq i}}$ is total, and that
it is correct about $B$ wherever it halts. (Note that showing that
$\psi^{X\oplus A_{\neq i}}$ is correct whenever it halts will also
show that it is not multivalued.)

\vspace{5pt}

To show that if $\psi^{X\oplus A_{\neq i}}(n)\downarrow$ then
$\psi^{X\oplus A_{\neq i}}(n)=B(n)$, let $Y$ be as above. Define
$q=\langle\tau,\epsilon\rangle$ where $\tau$ is defined as the finite
partial oracle for $A$ that agrees with $\sigma$ on $|\sigma|$, that
agrees with the portion of $Y$ that is queried in the computation of
$\phi^Y(n)$, and that ``halts late'' at all locations larger than
$|\sigma|$ if $Y$ was queried at that location, but $Y$ was not seen
to halt at that location.

Here, ``halting late'' means that $\tau$ halts at those locations, but
with an $l$ value larger than any $l$ value queried in the computation
of $\phi^Y(n)$. This ensures that $\tau$ agrees with the portion of
$Y$ that was queried while also having a large enough domain to not
violate the $\epsilon$ condition imposed by $p$.

Then $q\Vdash\phi^G(n)=\phi^Y(n)$, because $q$ agrees with the portion
of $Y$ used in the computation. But also that $q\leq p$, and so
$q\Vdash \phi^G(n)=B(n)$. Therefore $\phi^Y(n)=B(n)$. This proof was
shown for an arbitrary $Y$ as above, and so we have that if
$\psi^{X\oplus A_{\neq i}}(n)\downarrow$ then
$\psi^{X\oplus A_{\neq i}}(n)=B(n)$.

% \begin{itemize}
% \item $\tau\upharpoonright|\sigma|=\sigma$.
% \item For $m\geq|\sigma|$, if $m\in\dom(y)$, then $\tau

%   To show that if $\psi^{X\oplus A_{\neq i}}(n)\downarrow$ then
%   $\psi^{X\oplus A_{\neq i}}(n)=B(n)$, we show that any finite
%   initial segment of a $Y$ as in $\ast$ can be extended to a
%   condition $q\leq p$, so in particular, if $\phi^Y(n)=x$, then
%   there is a $q\leq p$ such that $q\Vdash\phi^G(n)=x$. But
%   $p\Vdash\phi^G(n)=B(n)$, so we have that $\phi^Y(n)=B(n)$.

\vspace{5pt}

To show that if $\psi^{X\oplus A_{\neq i}}$ is total, we show that
there is a generic generic oracle $G_0$ for $A$, extending $p$, whose
domain is contained in $\dom(\mathcal{F}(X))$, and so, for every $n$,
$G_0$ will be found as one of the $Y$ as above. From this, because
$p\Vdash \phi^G$ is total, we will have that for every $n$,
$\phi^{G_0}(n)\downarrow$, and so
$\psi^{X\oplus A_{\neq i}}(n)\downarrow$.

To show that there exists such a $G_0$, let $m_1$ be the largest
number such that $m_1\notin\dom(\mathcal{F}(X))$. Let $\tau$ be
defined as the finite partial oracle
$\mathcal{F}(X)\upharpoonright m_1+1$, and let
$q=\langle \tau,\epsilon\rangle$. By construction of $\mathcal{F}$, we
have that $q\leq p$. Let $G_0$ be any generic 
%%%  the phase generic generic is correct
%%%% do not remove. 
generic oracle for $A$
extending $q$. Then $G_0$ extends $p$, and its domain is contained in
$\dom(\mathcal{F}(X))$ because its domain restricted to $m_1+1$ is
equal to the domain of $\mathcal{F}(X)$ restricted to $m_1+1$, and
$\mathcal{F}(X)$ is a total oracle past $m_1+1$.
\end{proof}

\noindent{\bf Remark.} Combining Proposition \ref{1r}, and Theorem
\ref{nu1r}, with Corollaries 3.3, 3.11, and 3.14 from \cite{HJKS},
provides the following characterization of the level of randomness
required for to ensure quasiminimality in the uniform or nonuniform
cofinite, mod-finite, coarse, or generic degrees:

\begin{thm}
  In the uniform coarse and generic degrees, and also in the cofinite,
  and mod-finite degrees, every 1-random is quasiminimal.

  In the nonuniform coarse or generic degrees, every weakly 2-random
  is quasiminimal, but there exist 1-randoms (any 1-random that is
  also $\Delta^0_2$) which are not quasiminimal.

\end{thm}

In light of this, one might ask whether this provides a
characterization of the weakly 2-randoms, but it does not for a fairly
trivial reason:

\begin{obs}
  There exists a 1-random $A$ that is not weakly 2-random that is
  quasiminimal in both the nonuniform coarse and generic degrees.

\end{obs}

\begin{proof}[Proof (sketch).]
  Let $B$ be 1-random but not weakly 2-random, and let $C$ be weakly
  2-random relative to $B$. Let $A$ be the asymmetric join of $B$ and
  $C$ defined by
  $A=\{2^n:n\in B\}\cup\big(C\setminus\{2^n:n\in\omega\}\big)$.

  Note then that $A$ is 1-random but not weakly 2-random, but that $A$
  is coarsely (and generically) equivalent to $C$, and so quasiminimal
  in the nonuniform coarse (and generic) degrees.
\end{proof}

We observe now that our proof of Theorem \ref{nu1r} allows us to also
prove that 1-generics are quasiminimal in the nonuniform generic
degrees. The following analogue of Lemma \ref{HJKSlem} is proved
implicitly in the proof of Theorem 4.2 of \cite{HJKS}.

\begin{lem}[Hirschfeldt, Jockusch, Kuyper, Schupp
  \cite{HJKS}]\label{HJKSlem2}

  Assume $A$ is 1-generic, $B$ is noncomputable, and $k>1$. For each
  $i<k$, let $A_{=i}=\{n:kn+i\in A\}$, and let
  $A_{\neq i}=\bigoplus_{j\neq i} A_{=j}$.

  Then $(\exists i<k)(B\nleq_T A_{\neq i})$.

  Furthermore, for every $i$, $A_{=i}$ is 1-generic relative to
  $A_{\neq i}$.
\end{lem}

\begin{proof}[Proof (sketch)]
  A theorem of Yu \cite{Y} replaces the generalized form of Van
  Lambalgen's Theorem that is used in the proof of Lemma
  \ref{HJKSlem}, and $K$-triviality is not needed because if $A$ is
  1-generic relative to a noncomputable $B$, then $B\nleq_TA$.
\end{proof}

\begin{prop}[Cholak, Hirschfeldt, Igusa]\label{nu1g}
  Assume $A$ is 1-generic. Then $A$ is quasiminimal in the nonuniform
  generic degrees.

\end{prop}

\begin{proof}[Proof (sketch)] The proof is identical to the proof of
  Theorem \ref{nu1r}, using Lemma \ref{HJKSlem2} in place of Lemma
  \ref{HJKSlem}, and using Proposition \ref{1g} in place of
  Proposition \ref{1r} \end{proof}


\begin{thebibliography}{}

\bibitem{AHJ} E. Astor, D. Hirschfeldt, C. Jockusch, \emph{Dense
    computability, upper cones, and minimal pairs} (tentative title),
  in preparation.

\bibitem{DNWY} R. G. Downey, A. Nies, R. Weber, and L. Yu, Lowness and
  $\Pi^0_2$ nullsets, J. Symbolic Logic 71 (2006) 1044--1052.

\bibitem{DI} D. Dzhafarov and G. Igusa, \emph{Notions of robust
    information coding}, to appear.
  % Damir~D. Dzhafarov and Gregory Igusa, \emph{Notions of robust
  % information coding}, submitted.

\bibitem{HJKS} Denis Hirschfeldt, Carl Jockusch, Rutger Kuyper, and
  Paul Schupp, \emph{Coarse reducibility and algorithmic randomness},
  J. Symbolic Logic, to appear.

\bibitem{HJMS} D. Hirschfeldt, C. Jockusch, T. H. McNicholl, and
  P. Schupp, \emph{Asymptotic density and the coarse computability
    bound}, Computability, to appear.


\bibitem{HNS} D. R. Hirschfeldt, A. Nies, and F. Stephan, Using random
  sets as oracles, J. London Math. Soc. 75 (2007) 610--622.


\bibitem{I1} G. Igusa, \emph{Nonexistence of minimal pairs for generic
    computability}, J. Symbolic Logic, \textbf{78} (2013), no.~2,
  511--522.
  
\bibitem{I2} G. Igusa, \emph{The generic degrees of density-1 sets,
    and a characterization of the hyperarithmetic reals}, J. Symbolic
  Logic, to appear.

\bibitem{JS} C. Jockusch and P. Schupp, \emph{Generic computability,
    Turing degrees, and asymptotic density}, J. London Mathematical
  Society \textbf{85} (2012), no.~2, 472--490.
  % Carl~G. Jockusch, Jr. and Paul~E. Schupp, \emph{Generic
  % computability, {T}uring degrees, and asymptotic density},
  % J. Lond. Math. Soc. (2) \textbf{85} (2012), no.~2, 472--490.

\bibitem{KMSS} I. Kapovich, A. Miasnikov, P. Schupp, and V. Shpilrain,
  \emph{Generic-case complexity, decision problems in group theory,
    and random walks}, J. Algebra \textbf{264} (2003), no.~2,
  665--694.
  % Ilya Kapovich, Alexei Miasnikov, Paul Schupp, and Vladimir
  % Shpilrain, \emph{Generic-case complexity, decision problems in
  % group theory, and random walks}, J. Algebra \textbf{264} (2003),
  % no.~2, 665--694.

\bibitem{vL} M. van Lambalgen, \emph{The axiomatization of
    randomness}, J. Symbolic Logic 55 (1990) 1143--1167.

\bibitem{Y} L. Yu, \emph{Lowness for genericity}, Arch. Math. Logic 45
  (2006) 233--238.

\end{thebibliography}
\end{document}